\documentclass[11pt,a4paper]{amsart}
\usepackage[utf8]{inputenc}
\usepackage{mathrsfs}
\usepackage{mathdots}
\usepackage{crossreftools}
\usepackage{txfonts}
\usepackage{graphicx}
\usepackage{hyperref}
\usepackage{cleveref}
\usepackage{amssymb}
\usepackage{amsmath}
\usepackage{verbatim,enumerate}
\usepackage{tikz}
\usepackage{tikz-cd}
\usetikzlibrary{decorations.markings}
\usetikzlibrary{decorations.pathmorphing,arrows.meta}
\tikzset{double line with arrow/.style args={#1,#2}{decorate,decoration={markings,%
mark=at position 0 with {\coordinate (ta-base-1) at (0,1pt);
\coordinate (ta-base-2) at (0,-1pt);},
mark=at position 1 with {\draw[#1] (ta-base-1) -- (0,1pt);
\draw[#2] (ta-base-2) -- (0,-1pt);
}}}}
\usepackage{todonotes}
\usepackage[normalem]{ulem}

\numberwithin{equation}{section}
\numberwithin{figure}{section}

\newcommand{\C}{\mathbb{C}}
\newcommand{\R}{\mathbb{R}}
\newcommand{\Z}{\mathbb{Z}}

\newcommand{\GL}{\mathbf{GL}}
\newcommand{\Sp}{\mathbf{Sp}}
\newcommand{\SL}{\mathbf{SL}}

\newcommand{\Hom}{\mathrm{Hom}}
\newcommand{\Fix}{\mathrm{Fix}}
\newcommand{\Out}{\mathrm{Out}}
\newcommand{\Aut}{\mathrm{Aut}}

\newcommand{\Ga}{\Gamma}
\newcommand{\si}{\sigma}
\newcommand{\lra}{\longrightarrow}
\newcommand{\lmt}{\longmapsto}
\newcommand{\MB}{\mathbf{M}_B}
\DeclareMathOperator{\Dol}{Dol}
\newcommand{\MD}{\mathbf{M}_{\Dol}}
\DeclareMathOperator{\End}{End}
\DeclareMathOperator{\de}{deg}
\DeclareMathOperator{\rk}{rk}
\DeclareMathOperator{\Spec}{Spec}
\DeclareMathOperator{\Ext}{Ext}
\DeclareMathOperator{\Pic}{Pic}
\DeclareMathOperator{\Gal}{Gal}
\renewcommand{\phi}{\varphi}
\renewcommand{\rho}{\varrho}
\renewcommand{\leq}{\leqslant}
\renewcommand{\geq}{\geqslant}
\renewcommand{\H}{\mathbb{H}}

\DeclareMathOperator{\Sym}{Sym}
\DeclareMathOperator{\gr}{gr}

\DeclareMathOperator{\Imm}{Im}
\newcommand{\ov}[1]{\overline{#1}}

\counterwithout{figure}{section}

\newtheorem{theorem}{Theorem}[section]

\theoremstyle{definition}
\newtheorem{remark}[theorem]{Remark}
\newtheorem{prop}[theorem]{Proposition}
\newtheorem{definition}[theorem]{Definition}
\newtheorem{corollary}[theorem]{Corollary}

\newtheorem{lem}[theorem]{Lemma}

\title[Real Bialynicki-Birula flows]{Real Bialynicki-Birula flows in moduli spaces of\\ Higgs bundles}

\author{Florent Schaffhauser}
\address{Florent Schaffhauser, Institut f\"ur Mathematik, Universität Heidelberg, Im Neuenheimer Feld 205, 69120 Heidelberg, Germany}
\email{fschaffhauser@mathi.uni-heidelberg.de}

\author{Tommaso Scognamiglio}
\address{Tommaso Scognamiglio, Universitá di Bologna, Piazza di Porta S. Donato, 5, 40126 Bologna, Italy}
\email{scognamiglio@altamatematica.it}


\begin{document}

\begin{abstract}
Let $X$ be a compact Riemann surface $X$ of genus $\geq 2$ and let $\sigma:X \to X$ be an anti-holomorphic involution. Using real and quaternionic systems of Hodge bundles, we study the topology of the real locus $\R \MD(r,d)$ of the moduli space of semistable Higgs bundles of rank $r$ and degree $d$ on $X$, for the induced real structure $(E,\phi) \to (\sigma^*(\overline{E}),\sigma^*(\overline{\phi}))$. We show in particular that, when $\gcd(r,d)=1$, the number of connected components of $\R \MD(r,d)$ coincides with that of $\R \Pic_d(X)$, which is well-known.
\end{abstract}

\maketitle

\vspace{-15pt}

\tableofcontents

\vspace{-20pt}

\section{Introduction}

\subsection{Real structures} Let $X$ be a compact and connected Riemann surface of genus $g \geq 1$. A Higgs bundle on $X$ is a pair $(E,\phi)$ where $E$ is a holomorphic vector bundle over $X$ and $\phi:E \to E \otimes \Omega^1_X$ is a morphism of $\mathcal{O}_X$-modules. For all $r \in \Z_{>0}$ and $d \in \Z$, there exists quasi-projective algebraic variety $\MD(r,d)$, called the Dolbeault moduli space, which is the moduli space of semistable Higgs bundles over $X$ of rank $r$ and degree $d$. We denote by $\gcd(r,d)$ the greatest common divisor of $r$ and $d$. Recall that, if $\gcd(r,d)=1$, every $(E,\phi) \in \MD(r,d)$ is stable and the variety $\MD(r,d)$ is smooth. In general, the subvariety of stable points $\MD^{st}(r,d) \subseteq \MD(r,d)$ is the smooth locus of $\MD(r,d)$.

\smallskip

We are interested in the study of these spaces under the additional presence of a \textit{real structure}, i.e.\ an anti-holomorphic involution $\sigma:X \to X$, fixed once and for all. We will refer to the pair $(X,\sigma)$ as a \textit{Klein surface}. We denote indifferently by $X^{\sigma}$ or by $\R X$ the fixed-point set of $\si$ in $X$, also called the real locus of $X$. The real structure $\si$ on $X$ induces an anti-holomorphic involution 
\begin{equation}\label{Galois_action_Dolbeault}
\begin{array}{rcl}
\sigma: \MD(r,d) & \longrightarrow & \MD(r,d) \\
\big(E,\phi\big) & \longmapsto & \big(\sigma^*(\overline{E}),\sigma^*(\overline{\phi})\big) .
\end{array}
\end{equation} In particular, we are interested in the study of the topology of the set of real points $\R \MD(r,d)=\MD(r,d)^{\sigma}$, which is a real algebraic variety. These spaces are examples of \textit{branes} inside the Hitchin moduli space of harmonic bundles and we refer to \cite{BaSc1} for more details on this.

\smallskip

The study of the real points $\R \MD(r,d)$ is closely related to the understanding of \textit{real} and \textit{quaternionic} Higgs bundles on $X$. Recall that a real/quaternionic Higgs bundle is a triple $(E,\phi,\alpha)$ where $(E,\phi)$ is a Higgs bundle and $\alpha: E \to \sigma^*(\ov{E})$ is an isomorphism such that $\sigma^*(\overline{\alpha})\alpha=\pm\mathrm{Id}_E$. In particular, a real or quaternionic Higgs bundle $(E,\phi,\alpha)$ such that $(E,\phi)$ is semistable provides a real point $(E,\phi) \in \R\MD(r,d)$. As a partial converse, every \textit{stable} Higgs bundle $(E,\phi) \in \R \MD^{st}(r,d)$ can be endowed with either a real or a quaternionic structure $\alpha$ (but not both). Moreover, such an $\alpha$ is essentially unique (see \cref{realHiggsbundlemoduli} for details). We can therefore define a decomposition 
\begin{equation}
\label{decompositionintro1}
\R\MD^{st}(r,d)=\MD^{st,\R}(r,d) \bigsqcup \MD^{st,\H}(r,d)
\end{equation}
where the submanifolds $\MD^{st,\R}(r,d),\MD^{st,\H}(r,d)$ can be seen as moduli spaces of (geometrically stable) real/quaternionic Higgs bundles, respectively. In general, the points $\R \MD(r,d)$ do not have such a clear modular interpretation. For more on this, see \cref{polystable-real}.

\subsection{Bialynicki-Birula flows}

Scaling the Higgs field by a non-zero constant $t\in\C^*$ induces an action 
$t\cdot (E,\phi) \coloneqq (E, t\phi)$ of the multiplicative group $\C^*$ on the Dolbeault moduli space $\MD(r,d)$, which is best understood using the Hitchin fibration
$$
\begin{array}{rcl}
h : \MD(r,d) & \lra & \bigoplus_{i=1}^r H^0(X,\Omega_X^{\otimes i}) \\
(E,\phi) & \lmt & \text{coefficients of } \det(\phi-x\mathrm{Id}).
\end{array}
$$
Since the coefficients of the characteristic polynomial are elementary symmetric polynomials in the eigenvalues of $\phi$, the Hitchin fibration is $\C^*$-equivariant with respect to an appropriately scaled action of $\C^*$ on the symmetric differentials over $X$. As $\lim_{t\to 0} t\cdot \det(\phi-x\mathrm{Id})$ always exists in the vector space $\oplus_{i=1}^r H^0(X,\Omega_X^{\otimes i})$, the properness of $h$ (\cite[Theorem~6.11]{Simpson_IHES_2} implies that, for every $(E,\phi)\in\MD(r,d)$, the action map $\C^* \ni t \mapsto t \cdot (E,\phi)$ extends to a morphism of algebraic varieties $f_{(E,\phi)} : \C \to \MD(r,d)$. The Higgs bundle $(E_0,\phi_0) \coloneqq f_{(E,\phi)}(0)$ is usually denoted by $\lim_{t\to 0} t\cdot (E,\phi)$ and it is a fixed point of the $\C^*$-action. The $\C^*$-equivariance of $h$ implies that $h(E_0,\phi_0) = 0$ for every $\C^*$-fixed point in $\MD(r,d)$, so the fixed-point set $\MD(r,d)^{\C^*}$ is a closed subvariety of the so-called nilpotent cone $h^{-1}(0) \subset \MD(r,d)$, consisting of Higgs bundles $(E,\phi)$ with nilpotent Higgs field. As $h$ is proper, this also implies that $\MD(r,d)^{\C^*}$ is projective. We will call the map $f_{(E,\phi)} : \C \to \MD(r,d)$ the \textit{Bialynicki-Birula flow} of $(E,\phi)$ (or simply its BB flow). Often, we denote it simply by $f$. Since $\forall t \in\C^*, f(t) = t \cdot f(1)$, a BB flow is determined by its value at $1$.

\smallskip

The existence of Bialynicki-Birula flows for all $(E,\phi)\in\MD(r,)$, together with the fact that the fixed-point set of the $\C^*$-action is projective, makes $\MD(r,d)$ a semiprojective variety in the sense of \cite[Definition~1.1.1]{HRV}. By equivariance of the Hitchin fibration, if $\lim_{t\to\infty} t\cdot (E,\phi)$ exists, then $(E,\phi)$ is nilpotent. And by the properness of $h$, the converse is true. So a Higgs bundle $(E,\phi)$ is nilpotent if and only if its BB flow $f : \C \to \MD(r,d)$ extends to a morphism of algebraic varieties $\widehat{f} : \mathbf{P}^1 \to \MD(r,d)$. In particular, the nilpotent cone $h^{-1}(0)$ is the so-called \emph{core} of the semiprojective variety $\MD(r,d)$, which shows that $h^{-1}(0)$ is a deformation retract of $\MD(r,d)$ (\cite[Corollary~1.3.6]{HRV}). We refer to \cite{HRV,HauselHitchin} for more properties of the semiprojective variety $\MD(r,d)$.

\begin{definition}
Let $(E,\phi)\in\MD(r,d)$. The \textit{BB flow} of $(E,\phi)$ is the morphism of algebraic varieties $f : \C \to \MD(r,d)$ defined, for all $t\in\C^*$, by $f(t) = t \cdot (E,\phi)$. A BB flow is called \textit{complete} if it extends to a morphism of algebraic varieties $f : \mathbf{P}^1 \to \MD(r,d)$.
\end{definition}

As discussed above, the Higgs bundles with complete BB flows are exactly the nilpotent Higgs bundles. If we denote by $\mathcal{N}(r,d)$ the subvariety of $\MD(r,d)$ consisting of semistable Higgs bundles $(E,\varphi)$ with $\varphi=0$, then there is an inclusion $\mathcal{N}(r,d) \subseteq \MD(r,d)^{\C^*}$. Evidently, the variety $\mathcal{N}(r,d)$ can be identified with the moduli space of semistable vector bundles of rank $r$ and degree $d$. A characterisation of the fixed locus $\MD(r,d)^{\C^*}$ in terms of \textit{systems of Hodge bundles} was given by Simpson in \cite{Simpson_IHES_2},\cite{Simpson_LocalSystems}. Recall that a Higgs bundle $(E,\phi)$ is said to admit the structure of a system of Hodge bundles if there exists a decomposition into locally free $\mathcal{O}_X$-modules $E=\bigoplus_{p,q \in \mathbb{Z}}E^{p,q}$  such that $\phi(E^{p,q}) \subseteq E^{p-1,q+1} \otimes \Omega^1_X$. Using this description, Simpson showed that, for all $r$ and $d$, the moduli space $\MD(r,d)$ is connected (\cite[Corollary~11.10]{Simpson_IHES_2}. A key point in Simpson's proof is \cite[Lemma 11.9]{Simpson_IHES_2}, thanks to which he shows that any point $(E,\phi) \in \MD(r,d)$ can be connected to a point in $\mathcal{N}(r,d)$ through a finite sequence of BB flows, as illustrated in \cref{fig:BB_flows_to_Nrd}. We will recall the strategy of Simpson's proof for the connectedness of $\MD(r,d)$ in more detail in \cref{sectionHodgebundles}, but the point is that the connectedness of $\MD(r,d)$ follows from that of $\mathcal{N}(r,d)$, which was first proven by Narasimhan and Seshadri (\cite{NS}).

\smallskip

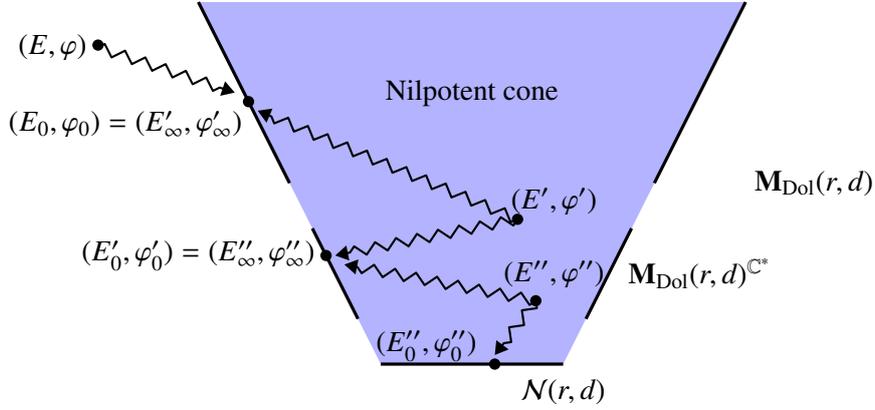
\begin{figure}[!h]
    \centering
    \begin{tikzpicture}[scale=1.2]
    \fill[blue!30] (-1,0) -- (-3,4) to (3,4) -- (1,0) -- cycle;
    
    \draw[very thick] (-1.75,1.5) -- (-1.25,0.5);
    \draw[very thick] (-2,2) -- (-3,4);
    \draw[very thick] (2,2) -- (3,4);
    \draw[very thick] (1.75,1.5) -- (1.25,0.5);
    \draw[very thick] (-1,0) -- (1,0);
    
    \node at (3.75,2) {$\mathbf{M}_{\mathrm{Dol}}(r,d)$};
    \node at (2.5,1) {$\mathbf{M}_{\mathrm{Dol}}(r,d)^{\mathbb{C}^*}$};
    \node at (0,3) {Nilpotent cone};
    \node at (1,-0.3) {$\mathcal{N}(r,d)$};
    
    \node[circle,fill,inner sep=1.5pt] at (-4.1,3.525) {};
    \node at (-4.6,3.5) {$(E,\phi)$};
            
    \node[circle,fill,inner sep=1.5pt] at (-2.45,2.9) {};
    \node at (-3.8,2.65) {$(E_0, \phi_0) = (E'_\infty, \phi'_\infty)$};

    \draw[-Triangle, thick, line join=round,
decorate, decoration={
    zigzag,
    segment length=8,
    amplitude=2,post=lineto,
    post length=2pt
}] 
        (-4.1,3.5) to (-2.6,3);
    
    \node[circle,fill,inner sep=1.5pt] at (0.5,1.6) {};
    \node at (0.9,1.8) {$(E',\phi')$};
    
    \draw[-Triangle, thick, line join=round,
decorate, decoration={
    zigzag,
    segment length=8,
    amplitude=2,post=lineto,
    post length=2pt
}] 
        (0.5,1.6) to (-2.35,2.8);
    
    \node[circle,fill,inner sep=1.5pt] at (-1.6,1.2) {};
    \node at (-3,1.2) {$(E'_0,\phi'_0) = (E''_\infty,\phi''_\infty)$};
    
    \draw[-Triangle, thick, line join=round,
decorate, decoration={
    zigzag,
    segment length=8,
    amplitude=2,post=lineto,
    post length=2pt
}] 
        (0.5,1.6) to (-1.5,1.2);
        
    \node[circle,fill,inner sep=1.5pt] at (0.7,0.7) {};
    \node at (0.9,1) {$(E'',\phi'')$};
    
    \draw[-Triangle, thick, line join=round,
decorate, decoration={
    zigzag,
    segment length=8,
    amplitude=2,post=lineto,
    post length=2pt
}] 
        (0.7,0.7) to (-1.4,1.1);

    \node[circle,fill,inner sep=1.5pt] at (0.25,0) {};
    \node at (-0.5,0.2) {$(E''_0,\phi''_0)$};
    
    \draw[-Triangle, thick, line join=round,
decorate, decoration={
    zigzag,
    segment length=8,
    amplitude=2,post=lineto,
    post length=2pt
}] 
        (0.7,0.7) to (0.25,0.1);

\end{tikzpicture}
    \caption{Bialynicki-Birula flows in the Dolbeault moduli space.}
    \label{fig:BB_flows_to_Nrd}
\end{figure}

In the real case, the space $\R\mathcal{N}(r,d)$ is not connected in general (\cite{Sch_JSG}) and neither is $\R\MD(r,d)$. Indeed, if for instance $\R X = \emptyset$, it can occur that the submanifolds $\MD^{\R}(r,d)$ and $\MD^{\H}(r,d)$ are both non-empty even if $\gcd(r,d) = 1$, so the open-closed decomposition (\ref{decompositionintro1}) implies that $\R \MD(r,d)$ has at least two connected components in this case. Similarly, if $\R X$ has $n > 0$ and $\gcd(r,d) =1$, then $\R \mathcal{N}(r,d)$ has $2^{n-1}$ connected components \cite{Sch_JSG}, so even if there are no quaternionic Higgs bundles, we cannot apply Simpson's proof directly. However, we are able to adapt it in order to count the connected components of the real algebraic variety $\R \MD(r,d)$. To outline our strategy, let us consider, for all $d \in \Z$, the moduli space of line bundles of degree $d$ on $X$, which we shall denote by $\Pic_d(X) \coloneqq \mathcal{N}(1,d)$. Its real locus $\R \Pic_d(X)$ is not connected in general and we refer to \cref{description-fibers} for a description of its connected components. Let $p:\R \Pic_d(X) \to \pi_0(\R \Pic_d(X))$ be the continuous map sending a line bundle $L \in \Pic_d(X)$ to the connected component of $\Pic_d(X)$ containing it. Denote by $c$ the composed map 
\begin{equation}
    \label{typeGaloisInvariantHiggsBundle}
    \begin{array}{rcl}
    c: \R \MD(r,d) & \longrightarrow & \pi_0\big(\R \Pic_d(X)\big) \\
    (E,\phi) & \longmapsto & p\big(\det(E)\big)
    \end{array}
\end{equation}
and, for every $s \in \pi_0(\R \Pic_d(X))$, put $\R \MD(r,d,s)\coloneqq c^{-1}(s)$. Our main result is then the following (see \cref{CC_real_locus} for the proof).

\begin{theorem}\label{main-result-intro}
For all $r,d$ such that $\gcd(r,d)=1$ and all $s \in \pi_0(\R \Pic_d(X))$, the space $\R\MD(r,d,s)\coloneqq c^{-1}(s)$ is connected.    
\end{theorem}

Theorem \ref{main-result-intro} is obtained by studying the properties of the BB flows in the presence of a real structure on the Dolbeault moduli space. Indeed, when $X$ is endowed with a real structure $\si$, the $\C^*$-action is compatible with the induced real structure \eqref{Galois_action_Dolbeault}, in the sense that $\sigma\big(t\cdot(E,\phi)\big) = \ov{t} \cdot \si(E,\phi)$. This shows that if $f:\C\to\MD(r,d)$ is a BB flow, then so is the map $f^\si(t) \coloneqq \si(f(\ov{t}))$ and this defines a Galois action on the set of BB flows. Another consequence of the compatibility relation is that there is an induced $\R^*$-action on $\R\MD(r,d)$ and that, by restriction to $\R$, the BB flow of a Galois-invariant Higgs bundle $(E,\phi)\in\R\MD(r,d)$ induces a morphism of real algebraic varieties $\R f \coloneqq f|_{\R} : \R \to \R \MD(r,d)$, which we call a \textit{real BB flow} (see \cref{limit-proposition}). Likewise, complete BB flows $\widehat{f} : (\mathbf{P}^1,\si) \to (\MD(r,d),\si)$ are real maps that induce morphisms $\R\widehat{f}: \R\mathbf{P}^1 \to \R \MD(r,d)$. Moreover, by definition of the map $c : \R \MD(r,d) \to \pi_0(\R \Pic_d(X))$, the $\R^*$-action on $\R\MD(r,d)$ preserves each fibre $\R \MD(r,d,s) \coloneqq c^{-1}(s)$. We will then show in \cref{CC_real_locus_section} that any point in $\R\MD(r,d,s)$ can be connected, via a finite sequence of real BB flows, to a point in $\R\mathcal{N}(r,d,s)$. More precisely, this will be obtained as an application of \cref{realHodgebundletheorem,flowrealpart-ss}. When $\gcd(r,d)=1$, the spaces $\R \mathcal{N}(r,d,s)$ are known to be connected (\cite[Theorem 3.8]{Sch_JSG} ), so Theorem \ref{main-result-intro} follows. As a corollary, we obtain the following explicit count of connected components of $\R\MD(r,d)$. This result also provides a modular interpretation of the connected components of $\R\MD(r,d)$ that extends the results of \cite{Sch_JSG} to the Higgs case (see \cref{Dol_CC_when_RX_ne,Dol_CC_when_RX_empty}).

\begin{theorem}\label{count_cc_intro}
    If $\gcd(r,d)=1$, the map $c : \R\MD(r,d) \to \pi_0(\R\Pic_d(X))$ introduced in \eqref{typeGaloisInvariantHiggsBundle} induces a bijection $\pi_0(\R\MD(r,d)) \simeq \pi_0(\R\Pic_d(X))$. In particular:
    \begin{enumerate}
        \item If $\R X$ has $n > 0$ connected components, then $\R\MD(r,d)$ has $2^{n-1}$ components.
        \item If $\R X = \emptyset$, then $\R\MD(2r'+1,2d')$ has $2$ connected components if $g$ is odd and $1$ if $g$ is even, while $\R\MD(r,2d'+1)$ has $1$ connected component if $r$ is odd and $g$ is even and $0$ if $r$ is even or $g$ is odd.
    \end{enumerate} 
\end{theorem}

If $\gcd(r,d)\neq 1$, we no longer have a proof that the spaces $\R\MD(r,d,s)$ are connected. However, our proof of Theorem \ref{main-result-intro} shows that, if $\R \mathcal{N}(r,d,s)$ happens to be connected, then the variety $\R \MD(r,d,s)$ is connected, too (\cref{implication-realpart}). In \cref{example-elliptic-real}, we show that this can occur when $X$ even if $\gcd(r,d)\neq1$.

\subsection{Real non-abelian Hodge correspondence and character varieties}

By the non\-abelian Hodge correspondence of Hitchin, Donaldson, Corlette and Simpson \cite{Simpson_LocalSystems}, there exists a homeomorphism 
$
\MD(r,d) \simeq \MB(r,d)
$
between the Dolbeault moduli space and the Betti moduli space, where the latter is defined as the $\GL_r(\C)$-character variety a certain central extension $\pi_d$ of $\pi_1X$ by $\Z$ (see \cref{CharVarSection} for details). While the $\C^*$-action and the Hitchin fibration cannot be constructed directly on the Betti side of the correspondence, our results on the topology of the real locus of $\MD(r,d)$ translate naturally to $\MB(r,d)$. Indeed, by Galois-equivariance of the nonabelian Hodge correspondence, the real structure $\si(E,\phi) = (\si^*(\ov{E}), \si^*(\ov{\phi}))$ of \eqref{Galois_action_Dolbeault} corresponds, on the Betti side, to the real structure $\beta_\si([\rho]) \coloneqq [\ov{\rho\circ\si_*}]$, which is well-defined on the character variety, even though $\si_*:\pi_1(X,x) \to \pi_1(X,\sigma(x))$ is not an involution in general (we refer to \cref{Galois_action_Betti} for a construction of the real structure $\beta_\si$). It is then immediate, when $\gcd(r,d)=1$ to deduce from \cref{count_cc_intro} a count of the connected components of $\R\MB(r,d)$ and we refer to \cref{CC_MBrd} for details. Moreover, we can give a modular interpretation of the connected components of $\R\MB(r,d)$ in terms of representations of orbifold fundamental groups. Namely, when $\gcd(r,d) = 1$, real points of the character variety of $\pi_d$ correspond to $\GL_r(\C)$-representations of $\pi_d$ that extend to representations of a certain \textit{non-central} extension of the orbifold fundamental group $\pi_1 (X/\Ga)$ by $\Z$ into either the semidirect product $\GL_r(\C) \rtimes \Ga$ or a non-trivial extension $\GL_r(\C) \times_\H \Ga$ (see \cref{orbifold_Betti}).

\smallskip

It is also possible to apply our results in a purely algebraic setting, to study the fixed-point set of the involutive automorphism
$\beta_{\widetilde{\si}}([\rho]) \coloneqq [\,^t(\rho\circ\si_*)^{-1}]$ of $\MB(r,d)$ obtained by replacing the anti-holomorphic involution of $\si(g) = \ov{g}$ of $\GL_r(\C)$ by the Cartan involution $\si(g) = \,^tg^{-1}$. This $\beta_{\widetilde{\si}}$ is now a \textit{holomorphic} involution of the Betti moduli space. On the Dolbeault side, however, it becomes the anti-holomorphic involution
\begin{equation}\label{alt_Galois_action_Dolbeault}
\widetilde{\sigma}: 
\begin{array}{rcl}
\MD(r,d) & \longrightarrow & \MD(r,d) \\
\big(E,\phi\big) & \longmapsto & \big(\sigma^*(\overline{E}),-\sigma^*(\overline{\phi})\big) .
\end{array}
\end{equation}
But the real structure \eqref{alt_Galois_action_Dolbeault} is conjugate to our original real structure \eqref{Galois_action_Dolbeault} via the order 4 automorphism $(E,\phi) \lmt (E,i\phi)$ of $\MD(r,d)$, as checked in \cref{conjugating_the_Galois_action_pf}. So the fixed-point set of the holomorphic involution $\beta_{\widetilde{\si}}$ is homeomorphic to the fixed-point set of the anti-holomorphic involution $\beta_\si$ in $\MB(r,d)$. On the Betti side, the modular interpretation of $\Fix(\beta_{\widetilde{\si}})$ is a direct analogue of that of $\Fix(\beta_\si)$, yielding the following result on twisted character varieties (see \cref{cxOrbifoldBetti} for details). 

\begin{theorem}
Let $\pi_d^\Ga$ be the orbifold fundamental group of the real Seifert bundle constructed in \eqref{realSeifertGroup_as_non_central_ext} using the real structure $\si: X \to X$ and let $\widetilde{\si}(g) = \,^t g^{-1}$ be the Cartan involution of $\GL_r(\C)$. If $\gcd(r,d) =1$ and $\R X \neq \emptyset$, then the twisted character variety
    $
    \Hom^\Z_\Ga(\pi_d^\Ga, \GL_r(\C) \rtimes_{\widetilde{\si}} \Gamma) / \GL_r(\C)
    $
    introduced in \cref{twistedCharVarSection} has $2^{n-1}$ connected components.
\end{theorem}

On the Dolbeault side, however, the real locus of the anti-holomorphic involution \eqref{alt_Galois_action_Dolbeault} does not seem to have a clear interpretation in terms of algebraic objects defined over $\R$.

\section{Real points of Dolbeault moduli spaces}

\subsection{Moduli spaces of Higgs bundles}
 
Let $X$ be a compact and connected Riemann surface of genus $g \geq 2$. A Higgs bundle on $X$ is a pair $(E,\phi)$ consisting of a holomorphic vector bundle $E$ on $X$ and a morphism of $O_X$-modules $\phi : E \to E \otimes \Omega^1_X$, known as a Higgs field. The degree  and rank of a Higgs bundle $(E,\phi)$ are the degree and rank of the underlying vector bundle $E$. A morphism of Higgs bundles $f:(E_1,\phi_1) \to (E_2,\phi_2)$ is a morphism of vector bundles $f:E_1 \to E_2$ that makes the following diagram commute.
$$
\begin{tikzcd}
    E_1 \ar[d,"f"'] \ar[r, "\phi_1"'] & E_1 \otimes \Omega^1_X\ar[d, "f \otimes \mathrm{Id}_{\Omega^1_X}"']  \\
    E_2 \ar[r, "\phi_2"'] & E_2 \otimes \Omega^1_X
\end{tikzcd}
$$
Recall that the slope $\mu(E)$ of a non-zero vector bundle $E$ is defined as $\mu(E) \coloneqq \frac{\de(E)}{\rk(E)}$. We say that a Higgs bundle $(E,\phi)$ is semistable if, for every non-trivial $\phi$-invariant subbundle $F \subseteq E$, we have $\mu(F) \leq \mu(E)$. The Higgs bundle $(E,\phi)$ is said to be stable if the previous inequality is strict for every $F$. If $\gcd(r,d) = 1$, a Higgs bundle $(E,\Phi)$ is semistable  if and only if it is stable. For a stable Higgs bundle $(E,\phi)$, there is an equality $\End((E,\phi))=\{\lambda \mathrm{Id}_E, \lambda \in \C\}$, and for every semistable Higgs bundle $(E,\phi)$, there exists a filtration $0=E_{\ell} \subseteq E_{\ell-1} \subseteq \cdots \subseteq E_0=E$ by $\phi$-invariant subbundles (meaning that $\phi(E_j) \subset E_j \otimes \Omega^1_X$) such that, for each $j=1,\dots,\ell$, we have $\mu\left(E_{j-1}/E_j\right)=\mu(E)$ and $(E_{j-1}/E_j,\phi)$ is stable. We denote by $(\gr(E),\gr(\phi))$ the associated graded object 
$
\big(\gr(E),\gr(\phi)\big)\coloneqq (E_\ell,\phi) \oplus \cdots \oplus (E/E_1,\phi)
$.
Its graded isomorphism class depends only on the semistable Higgs bundle $(E,\phi)$, not on the choice of the filtration satisfying the previous conditions. The Higgs bundle $(E,\phi)$ is called polystable if it is semistable and $(E,\phi) \simeq (\gr(E),\gr(\phi))$. Moreover, two semistable Higgs bundles $(E,\phi),(E',\phi')$ are called $S$-equivalent if 
$
(\gr(E),\gr(\phi)) \simeq (\gr(E'),\gr(\phi')).
$
For all $r,d$, there exists a quasi-projective complex variety $\MD(r,d)$  of dimension $2(r^2(g-1) + 1)$ parameterizing $S$-equivalence classes of semistable Higgs bundles of rank $r$ and degree $d$ over $X$. Following Simpson, we will refer to these varieties as Dolbeault moduli spaces. If $\gcd(r,d)=1$, the moduli space $\MD(r,d)$ is smooth, see for instance \cite[Proposition 7.4]{Nitsure}. For all $r$ and $d$, there is an open smooth subvariety $\MD^{st}(r,d) \subseteq \MD(r,d)$, whose points are isomorphism classes of stable Higgs bundles of rank $r$ and degree $d$, and a subvariety $\mathcal{N}(r,d) \subseteq \MD(r,d)$, projective of dimension $r^2(g-1) + 1$ and consisting of semistable Higgs bundles $(E,\phi)$ such that $\phi=0$. The subvariety $\mathcal{N}(r,d)$ will be identified with the moduli space of semistable vector bundles of rank $r$ and degree $d$ over $X$. It contains an open smooth subvariety $\mathcal{N}^{st}(r,d) \subseteq \mathcal{N}(r,d)$, whose points are isomorphism classes of stable vector bundles of rank $r$ and degree $d$.

\subsection{The Galois action on Dolbeault moduli spaces}

Let $\sigma:X \to X$ be an anti-holomorphic involution of $X$, also called a real structure. The involution $\sigma$ determines a real projective curve $X_{\R}$, i.e.\ a smooth projective variety of dimension $1$ over $\Spec(\R)$ such that $X_{\R} \times_{\Spec(\R)} \Spec(\C) \simeq X$ and, via the isomorphism above, $\sigma$ corresponds to complex conjugation. The isomorphism above induces an identification $X_{\R}(\R)=X^{\sigma}$. We will denote by $n \geq 0$ the number of connected components of the fixed point set $X^{\sigma}$. Recall that each connected component of $X^{\sigma}$ is homeomorphic to a circle. In what follows, we will denote by $\Gamma\coloneqq \Gal(\C/\R)=\{1,\sigma\}$ the absolute Galois group of $\R$. We will also use the notation $\R X \coloneqq X^{\sigma}$. The group $\Gamma$ acts on $\MD(r,d)$ via $\sigma\big(E,\phi\big)=\big(\sigma^*(\overline{E}),\sigma^*(\overline{\phi})\big)$. This is an anti-holomorphic involution, or real structure, on $\MD(r,d)$, whose set of fixed points will be denoted by $\R \MD(r,d)$. If $\gcd(r,d) = 1$, it is naturally a real analytic manifold of real dimension $2r^2(g-1) + 1$. Our goal in this article is to study the topology of the space $\R \MD(r,d)$, in particular its connected components and their relation with $\R \mathcal{N}(r,d)$. In order to do so, we start by a quick review of the modular interpretation of the varieties $\R\MD(r,d)$, $\R\MD^{st}(r,d)$, $\R \mathcal{N}(r,d)$ and $\R \mathcal{N}^{st}(r,d)$.

\subsubsection{Real and quaternionic Higgs bundles}
\label{realmodular}

For a holomorphic vector bundle $E \to X$, we will say that an anti-holomorphic map $\tau: E \to E$ covers $\sigma$ if $\tau$ is fibrewise $\C$-antilinear and makes the following diagram commute.
$$
\begin{tikzcd}
E \ar[r,"\tau"'] \ar[d] &E \ar[d] \\
X \ar[r,"\sigma"'] &X
\end{tikzcd}
$$
Note that giving an anti-holomorphic map $\tau:E\to E$ covering $\sigma$ is equivalent to giving a morphism of holomorphic vector bundles $\alpha:E \to \sigma^*(\overline{E})$. Moreover, $\tau^2 = \pm \mathrm{Id}_E$ if and only if $\sigma^*(\overline{\alpha}) \alpha = \pm \mathrm{Id}_E$. The involution $\sigma$ induces a canonical antihomolorphic involution on $\Omega^1_X$ covering $\sigma:X \to X$, which we still denote by $\sigma:\Omega^1_X \to \Omega^1_X$. A real Higgs bundle is a triple $(E,\phi,\tau)$, where $(E,\phi)$ is a Higgs bundle and $\tau:E \to E$ is an anti-holomorphic involution covering $\sigma:X \to X$ and making the following diagram commute:
\begin{equation}
\label{diagramreal}
\begin{tikzcd}
E \ar[r, "\phi"'] \ar[d, "\tau"'] &  E \otimes \Omega^1_X  \ar[d, "\tau \otimes \sigma"'] \\
E \ar[r, "\phi"'] & E \otimes \Omega^1_X
\end{tikzcd}
\end{equation}
Quaternionic Higgs bundles are defined similarly, except that $\tau^2 = - \mathrm{Id}_E$. A morphism of real/quaternionic Higgs bundles $f : (E_1,\phi_1,\tau_1) \to (E_2,\phi_2,\tau_2)$ is a morphism of Higgs bundles $(E_1,\phi_1) \to (E_2,\phi_2)$ which commutes with $\tau_1$ and $\tau_2$. 

\smallskip

If $(E,\phi,\tau)$ is a real Higgs bundle, the involution $\tau$ endows $E$ with a structure of real algebraic vector bundle over $X$, meaning that there exists an algebraic vector bundle $E_{\R}$ over $X_{\R}$ such that we have an isomorphism $E_{\R} \times_{\Spec(\R)} \Spec(\C) \simeq E$
and, via this isomorphism, $\tau$ corresponds to the complex conjugation on the second factor. Note that giving such a map $\tau$ is equivalent to giving an isomorphism of Higgs bundles $\alpha : \big(E,\phi\big) \to \big(\sigma^*(\overline{E}),\sigma^*(\overline{\phi})\big)$ with the property that $\sigma^*(\overline{\alpha})\alpha= \mathrm{Id}_{E}$. If $(E,\phi,\tau)$ is a real Higgs bundle whose underlying Higgs bundle $(E,\phi)$ is semistable (of rank $r$ and degree $d$, say), then the corresponding point $(E,\phi) \in \MD(r,d)$ belongs to $\R\MD(r,d)$. However, not all real points of $\MD(r,d)$ come from real Higgs bundles. For instance, a quaternionic Higgs bundle whose underlying Higgs bundle is semistable also defines a real point of the Dolbeault moduli space.

\begin{remark}
If $\R X \neq \emptyset$ and $(E,\phi,\tau)$ is a quaternionic vector bundle over $X$, then $\rk(E)$ must be even. Indeed, taking a point $x \in \R X$, the map $\tau$ restricts to an antilinear map on the fiber $\tau_x :E_x \to E_x$ such that $\tau_x^2=-\mathrm{Id}$, i.e.\ to a quaternionic structure on the $\C$-vector space $E_x$, and having such a structure is only possible if the dimension of $E_x$ is even.
\end{remark}

\subsubsection{Real locus of the Dolbeault moduli space}

We then have the following results, on the structure of the set of real points of the stable locus of the Dolbeault moduli space. We refer to \cite[Proposition 3.8]{Sch_JDG} or \cite[Proposition 2.8]{Sch_JSG} for a proof (the arguments of \cite{Sch_JDG, Sch_JSG} are stated for Galois-invariant stable vector bundles but translates \textit{verbatim} to the case of Higgs bundles because the automorphism group of a stable Higgs bundle also reduces to $\C^*$). 

\begin{prop}
\label{realHiggsbundlemoduli}
Let $(E,\phi)$ be a stable Higgs bundle such that $(\sigma^*(\ov{E}), \sigma^*\ov{\phi}) \simeq (E,\phi)$, Then $(E,\phi)$ admits either a real or a quaternionic structure, but not both. Moreover, such a structure is unique up to isomorphism of real/quaternionic Higgs bundle.
\end{prop}

We can therefore define a decomposition 
\begin{equation}
\label{decompositionDol}
\R\MD^{st}(r,d)=\MD^{st,\R}(r,d) \bigsqcup \MD^{st,\H}(r,d)
\end{equation}
where $\MD^{st,\R}(r,d)$ corresponds to the stable Higgs bundle that can be endowed with a real structure and $\MD^{st,\H}(r,d)$ to the ones that can be endowed with a quaternionic structure. In particular, if $\gcd(r,d) = 1$, we have a well-defined decomposition $\R\MD(r,d)=\MD^{\R}(r,d) \bigsqcup \MD^{\H}(r,d)$. Similarly, we have a decomposition for stable vector bundles $\R \mathcal{N}^{st}(r,d)=\mathcal{N}^{st,\R}(r,d) \sqcup \mathcal{N}^{st,\H}(r,d) $ depending on whether they can be endowed with a real or quaternionic structure. By the uniqueness part of \cref{realHiggsbundlemoduli}, the submanifolds $\MD^{st,\R}(r,d)$ and $\MD^{st,\H}(r,d)$ can be seen as moduli spaces of geometrically stable real and quaternionic Higgs bundles (see \cref{realstable-not-geom}). If $\R X \neq \emptyset$, the decomposition in \cref{decompositionDol} can be refined as follows. Consider a point $(E,\phi) \in \MD^{st,\R}(r,d)$ and fix a real structure $\tau:E \to E$. Taking real points, we obtain a real vector bundle $\R E$ over the real manifold $\R X$, which is a union of $n > 0$ circles. The first Stiefel-Whitney class $w_1(\R E)$ therefore determines an element of $H^1(\R X,\Z/2\Z) \simeq (\Z/2\Z)^n$ and it can be checked that, if $w_1(\R E)=(s_1,\dots,s_n) \in (\Z/2\Z)^n ,$ then $s_1+\ldots+s_n=d \ \text{mod} \ 2$ (see for instance \cite{BHH}). We thus have a decomposition \begin{equation}
\label{decompositionDol1} 
\MD^{st,\R}(r,d) \quad =\bigsqcup_{\substack{s \in (\Z/2\Z)^n \\ s_1+\cdots+s_n=d \ \text{mod} \ 2}} \MD^{st}(r,d,s)
\end{equation}
where $\MD^{st}(r,d,s)$ stands for the space of real Higgs bundles $(E,\phi) \in \MD^{st,\R}(r,d)$ with associated Stiefel-Whitney class $w_1(E_{\R}(\R))=s$.

\begin{remark}
\label{realstable-not-geom}
For a real/quaternionic Higgs bundle $(E,\phi,\tau)$ there is a natural notion of (semi)stability, namely that for each subbundle $F \subseteq E$ which is both $\phi$-invariant and $\tau$-invariant, we have $\mu(F) (\leq) \mu(E)$. Adapting the argument of \cite[Proposition 2.7]{Sch_JSG} to the Higgs setting, we see that a stable real/quaternionic Higgs bundle $(E,\phi,\tau)$ satisfies one of the following two mutually exclusive conditions:
\begin{itemize}
    \item The underlying Higgs bundle $(E,\phi)$ is stable, in which case we will say that $(E,\phi,\tau)$ is \textit{geometrically stable}.
    \item The underlying Higgs bundle $(E,\phi)$ is not stable, in which case there exists a stable Higgs bundle $(\mathcal{E},\Phi)$ such that $(\mathcal{E},\Phi) \not\simeq (\sigma^*(\overline{\mathcal{E}}),\sigma^*(\overline{\Phi}))$ and an isomorphism of Higgs bundles
    \begin{equation}
    \label{notgeomstable}
        (E,\phi) \simeq (\mathcal{E},\Phi) \oplus (\sigma^*(\overline{\mathcal{E}}),\sigma^*(\overline{\Phi})\big)
    \end{equation} with real/quaternionic structure $\tau$ given by exchanging the factors in (\ref{notgeomstable}). We refer to \cite[Proposition 2.7]{Sch_JSG} for the details of the proof. 
\end{itemize}
Note that real/quaternionic Higgs bundles are semistable in the real/quaternionic sense if and only if they are semistable in the complex sense, and that stable real/quaternionic Higgs bundles are always at least polystable in the complex sense.
\end{remark}

If $\gcd(r,d)\neq 1$, then not all real points $(E,\phi) \in \R\MD(r,d)$ admit such a modular interpretation. We can, however, still give the following description of polystable bundles in the real locus, as discussed in \cite[Section 2.5]{Sch_JSG}.

\begin{prop}
\label{polystable-real}
Given a polystable Higgs bundle $(E,\phi)\in \R \MD(r,d)$, there  exist stable real or quaternionic bundles $(E_1,\phi_1,\tau_1),$ $\dots,(E_\ell,\phi_\ell,\tau_\ell)$, all of them of slope $\frac{d}{r}$, and a graded isomorphism of Higgs bundles $(E,\phi) \simeq (E_1,\phi_1) \oplus \cdots \oplus (E_\ell,\phi_\ell)$.    
\end{prop}

\subsection{Topological type of real points}
\label{description-fibers}
 
Let us define a morphism of real algebraic varieties 
\begin{equation}\label{obstruction_map}
    c : \R\MD(r,d) \to \pi_0\big(\R\Pic_d(X)\big)\, ,
\end{equation} where $\Pic_d(X)$ is the connected component of the Picard group of $X$ parameterizing line bundles of degree $d$ on $X$, by composing the determinant morphism
$$
\det : 
\begin{array}{rcl}
    \R\MD(r,d) & \longrightarrow & \R\Pic_d(X) \\
    (E, \phi) & \longmapsto & \det E
\end{array}
$$
with the canonical continuous map $\R\Pic_d(X) \to \pi_0(\R\Pic_d(X))\,.$ Recall that the finite set $\pi_0(\R \Pic_d(X))$ has been described in \cite{GH}. For every $s \in \pi_0(\R\Pic_d(X))$, we put $\R\MD(r,d,s)\coloneqq c^{-1}(s)$. Since $\pi_0(\R\Pic_d(X))$ is discrete, we have an open-closed decomposition
\begin{equation}
\label{fibers-topologicaltype}
\R\MD(r,d) = \bigsqcup_{s\in\pi_0(\R\Pic_d(X))} \R\MD(r,d,s)\,.
\end{equation}
We will show in \cref{CC_real_locus} that, when $\gcd(r,d) = 1$, these fibres are, in fact, connected, so the above decomposition gives the connected components of the real locus of the Dolbeault moduli space. For the proof, it will be useful to have a modular interpretation of each space $\R\MD(r,d,s)$, which we now give. In particular, we will relate these fibers to the subspaces introduced in \cref{decompositionDol}.

\subsubsection{Case 1}\label{modular_interp_real_locus_with_real_pts_on_curve} 

Assume first that $\R X \neq \emptyset$, with $n > 0$ connected components. Then $\R\Pic_d(X)$ consists of isomorphism classes of real line bundles and it has $2^{n-1}$ connected components (\cite{GH}). These connected components are distinguished by the first Stiefel-Whitney class $s \coloneqq w_1(\R L) \in H^1(\R X; \Z/2\Z) \simeq (\Z / 2\Z)^n$ of a real line bundle $L \in \R Pic_d$, which is subject to the condition $s_1 + \ldots + s_n = d\ \mathrm{mod}\,{2}$. More precisely, the continuous map
$$
\begin{array}{rcl}
    \R\Pic_d(X) & \longrightarrow & (\Z/2\Z)^n  \\
    L & \longmapsto & s = (s_1, \ldots, s_n) \coloneqq w_1(\R L)
\end{array}
$$ induces a bijection 
$
\pi_0(\R\Pic_d(X)\big) 
\simeq 
\{ s \in (\Z/2\Z)^n\ |\ s_1 +\, \ldots\, + s_n = d\ \text{mod}\ 2\} \simeq (\Z/2\Z)^{n-1}.
$ So the morphism $c$ defined in \eqref{obstruction_map} can be viewed in this case as a continuous map
$$
c : 
\begin{array}{rcl}
    \R\MD(r,d) & \longrightarrow & (\Z/2\Z)^n \\
    (E, \phi) & \longmapsto & w_1(\R\det E)
\end{array}
$$
We will now study the intersection of the fibres of $c$ with $\MD^{st}(r,d)$. 

\smallskip

\textbf{Subcase 1.} Let us first assume that $r = 2r' +1$. In particular, since $n > 0$, there are no quaternionic vector bundles of rank $2r'+1$ on $X$. Then, comparing \cref{decompositionDol} and \cref{decompositionDol1}, we have, for all $s \in (\Z/ 2\Z)^n$, $\R\MD^{st}(2r'+1, d,s) = \MD^{st, \R}(2r'+1, d, s)$ and this set is non-empty if and only if $s_1 + \ldots + s_n = d\ \mathrm{mod}\,2$.

\smallskip

\textbf{Subcase 2.} Let us now assume that $r = 2r'$. In this case, there can be quaternionic vector bundles of rank $2r'$, provided the degree $d$ is also even (since the rank and degree of a quaternionic bundle must satisfy $d + r(g-1) \equiv 0 \mod 2$; see for instance \cite{BHH}). In particular, when $n > 0$ and we also assume that $\gcd(r,d) =1$, then there are in fact no quaternionic vector bundles on $X$. In general, though, the situation is the following.

\begin{enumerate}
    \item If $d = 2d'+1$, then, for all $s \in (\Z/ 2\Z)^n$, we have $ \R\MD^{st}(2r',2d'+1,s) = \MD^{st, \R}(2r', 2d'+1,s)$ and the latter is non-empty if and only if $s_1 + \ldots + s_n = 1\ \mathrm{mod}\,2$.
    \item If $d = 2d'$, then the determinant line bundle $L \coloneqq \det(E)$ of a quaternionic vector bundle $(E,\tau_\H)$ is a real line bundle which satisfies $\R L = \R\det E = \R\det(E|_{\R X})$, where $E|_{\R X}$ is the restriction to $\R X$ of the complex vector bundle $E$. Note that $E|_{\R X}$ is complex vector bundle with quaternionic structure over a space with trivial real structure \cite{Atiyah_real_structures}. Consequently, the vector bundle $E|_{\R X}$ admits a reduction of structure group to the symplectic group $\Sp(2r';\C) \subset \SL(2r';\C)$ and this reduction of structure group is compatible with the real structure. So $\det(E|_{\R X})$ is trivial as a complex line bundle with real structure, hence $w_1(\R L) = 0$. This shows that, when $n > 0$, $r = 2r'$ and $d = 2d'$, we have a decomposition $\R\MD^{st}(2r', 2d',0) = \MD^{st,\R}(2r',2d',0) \sqcup \MD^{st, \H}(2r', 2d')$ and, for all $s \neq 0$, $\R\MD^{st}(2r', 2d',s) = \MD^{st, \R}(2r', 2d',s)$ and the latter is non-empty if and only if $s_1 + \ldots + s_n = 0\ \mathrm{mod}\,2$.
\end{enumerate}

We will prove in \cref{CC_real_locus} that, if $n > 0$ and $\gcd(r,d) = 1$, then for all $s \in (\Z/2\Z)^n$ such that $s_1 + \ldots + s_n = d\ \mathrm{mod}\,2$, the fibre $ \R \MD(r,d,s) = c^{-1}(s)$ is non-empty and connected, which will yield the following result.

\begin{theorem}\label{Dol_CC_when_RX_ne}
Assume that $n > 0$ and that $\gcd(r,d) = 1$. Then the real locus of the moduli space of Higgs bundles of rank $r$ and degree $d$ has $2^{n-1}$ connected components. More precisely, we have the following decomposition into connected components
$$\R\MD(r,d) \quad = \bigsqcup_{\substack{s \in (\Z/2\Z)^n \\ s_1+\cdots+s_n=d \ \mathrm{mod} \ 2}} \MD^{st, \R}(r,d,s)$$ 
where $\MD^{st,\R}(r,d,s) = c^{-1}(s)$ can be viewed as the moduli space of geometrically stable real Higgs bundles $(E,\phi,\tau)$ of rank $r$, degree $d$ and such that $w_1(\R E) = s$.
\end{theorem}

\subsubsection{Case 2}\label{modular_interp_real_locus_no_real_pts_on_curve} Assume now that $\R X = \emptyset$. When $X$ has no real points, real vector bundles must have even degree, while quaternionic vector bundles can now occur in odd rank (subject to the condition $d + r(g-1) \equiv 0\ \mathrm{mod}\,2$). The real locus of the Picard variety $\Pic_d(X)$ is given as follows \cite{GH}:
\begin{itemize}
\item $\R\Pic_{2d'}(X) = \Pic_d(X)^{\R}$ if $g$ is even.
\item $\R\Pic_{2d'}(X) = \Pic_d(X)^{\R} \sqcup \Pic_{2d'}^{\H}$ if $g$ is odd.
\item $\R\Pic_{2d'+1}(X) = \Pic_{2d'+1}^{\H}$ if $g$ is even.
\item $\R\Pic_{2d'+1}(X) = \emptyset$ if $g$ is odd.
\end{itemize}
So, when $\R\Pic_d(X)\neq\emptyset$, the group $\pi_0(\R\Pic_d(X))$ is either trivial or isomorphic to $\Z/2\Z$. We denote its neutral element by $\R$ and, when there is a non-trivial element, we denote it by $\H$. As a consequence, when $n=0$, the decomposition in \cref{decompositionDol} becomes the following:

\begin{enumerate}
    \item If $r = 2r' + 1$, $d = 2d'$ and $g = 2g'$, then $\R\MD^{st}(2r'+1,2d') = \MD^{st,\R}(2r'+1,2d')$ consists of real Higgs bundles only and coincides with the intersection of the fibre $c^{-1}(\R)$ with the stable locus $\MD^{st}(2r'+1,2d')$ (recall that $\pi_0(\R\Pic_{2d'})$ is trivial when $g$ is even, because there are only real line bundles in the real locus of the Picard variety in this case).
    \item If $r = 2r' + 1$, $d = 2d'$ and $g = 2g' + 1$, then $\R\MD^{st}(2r'+1,2d') = \MD^{st,\R}(2r'+1,2d') \sqcup \MD^{st,\H}(2r'+1,2d')$ consists of real and quaternionic Higgs bundles. Moreover, we have $\MD^{st,\R}(2r'+1,2d') = c^{-1}(\R) \cap \MD^{st}(2r'+1,2d')$ and $\MD^{st,\H}(2r'+1,2d') = c^{-1}(\H) \cap \MD^{st}(2r'+1,2d')$.
    \item If $r = 2r' + 1$, $d = 2d' + 1$ and $g = 2g'$, then $\R\MD^{st}(2r'+1,2d'+1) = \MD^{st,\H}(2r'+1,2d'+1)$ consists of quaternionic Higgs bundles only and coincides with the intersection of the fibre $c^{-1}(\R)$ with the stable locus of the Dolbeault moduli space (recall that $\pi_0(\R\Pic_{2d'+1})$ is trivial when $g$ is even, because there are only quaternionic line bundles in the real locus of the Picard variety in this case; in particular, even though there are only quaternionic bundles in $\R\MD^{st}(2r'+1,2d'+1)$, there is no fibre $c^{-1}(\H)$ here).
    \item If $r = 2r' + 1$, $d = 2d' + 1$ and $g = 2g' + 1$, then $\R\MD^{st}(2r'+1,2d'+1) = \emptyset$.
    \item If $r = 2r'$ and $d = 2d' + 1$, then, for all $g\geq 1$, $\R\MD^{st}(2r',2d'+1) = \emptyset$.
    \item If $r = 2r'$ and $d = 2d'$, then, regardless of the parity of $g$, $\R\MD^{st}(2r',2d') = \MD^{st,\R}(2r',2d') \sqcup \MD^{st,\H}(2r',2d')$ consists of real and quaternionic Higgs bundles. In this case, the group $\pi_0(\R\Pic_{2d'})$ is trivial, regardless of the parity of $g$, so the above is a decomposition of $c^{-1}(\R) \cap \MD^{st}(2r',2d')$. Note that this case does not occur when $\gcd(r,d) = 1$.
\end{enumerate}

We will prove in \cref{CC_real_locus} that, if $n = 0$ and $\gcd(r,d) = 1$, then for all $s\in\pi_0(\R\Pic_d(X))$, the fibre $\R \MD(r,d,s) \coloneqq c^{-1}(s)$ is connected, which will yield the following result.

\begin{theorem}
\label{Dol_CC_when_RX_empty}
Assume that $n=0$ and that $\gcd(r,d) = 1$. Then, the real locus of the moduli space of Higgs bundles of rank $r$ and degree $d$ admits the following decomposition into connected components:
\begin{enumerate}
    \item If $r = 2r' + 1$, $d = 2d'$ and $g = 2g'$, then $\R\MD(2r' + 1, 2d') = \MD^{\R,st}(2r' + 1, 2d')$. 
    \item If $r = 2r' + 1$, $d = 2d'$ and $g = 2g' + 1$, then $\R\MD(2r' + 1, 2d') = \MD^{\R,st}(2r' + 1, 2d') \sqcup \MD^{\H,st}(2r' + 1, 2d')$. 
    \item If $r = 2r' + 1$, $d = 2d' + 1$ and $g = 2g'$, then $\R\MD(2r'+1 , 2d'+1) = \MD^{\H,st}(2r' + 1, 2d')$. 
    \item If $r = 2r' + 1$, $d = 2d' + 1$ and $g = 2g' + 1$, then $\R\MD(2r'+1 , 2d'+1) = \emptyset$.
    \item If $r = 2r'$ and $d = 2d'+1$, then, for all $g\geq 1$, $\R\MD(2r,2d'+1) = \emptyset$.
\end{enumerate}
\end{theorem}

\begin{remark}
\label{decompositon-openclosed-picard}
Notice that the description of the spaces $\R\MD(r,d,s)$ given in \cref{Dol_CC_when_RX_ne,Dol_CC_when_RX_empty} implies that, when $\gcd(r,d)=1$, the decompositions (\ref{decompositionDol}),(\ref{decompositionDol1}) are open-closed since they correspond to the decomposition in fibers (\ref{fibers-topologicaltype}).
\end{remark}

\section{Real Bialynicki-Birula flows}
\label{sectionHodgebundles}

\subsection{Real and quaternionic Hodge bundles}

Recall from \cite{Simpson_LocalSystems} that a Higgs bundle $(E,\phi)$ admits the structure of a system of Hodge bundles if there exists a decomposition into locally free $O_X$-modules $E=\bigoplus_{p,q \in \mathbb{N}}E^{p,q}$ such that $\phi(E^{p,q}) \subseteq E^{p-1,q+1} \otimes \Omega^1_X$. They were used by Simpson to give the following description of the fixed points of the $\C^*$-action on $\MD(r,d)$.

\begin{theorem}\cite[Lemma 4.1]{Simpson_LocalSystems}
\label{fixedpointsHodge}
Let $(E, \phi) \in \MD(r,d)$ be a polystable Higgs bundle. Then the following conditions are equivalent:
\begin{enumerate}
    \item For all $t \in \C^*$, $t\cdot(E,\phi) \simeq (E,\phi)$.
    \item There exists $t\in \C^*\setminus\{\text{roots of unity}\}$, such that $t\cdot (E,\phi)\simeq (E,\phi)$.
    \item $(E,\phi)$ is isomorphic to a system of Hodge bundles.
\end{enumerate}
If $(E,\phi)$ is stable, the system of Hodge bundles to which it is isomorphic is unique up to permutation of the indices.
\end{theorem}

This description of the $\C^*$-fixed points is key to Simpson's proof of the connectedness of variety $\MD(r,d)$. In the present section, we adapt his methods to describe the connected components of the variety $\R\MD(r,d)$ if $\gcd(r,d) = 1$ (i.e.\ in the smooth case). But first we recall the following fact.
\begin{remark}
\label{stable-factors}
If $(E,\phi)\in \MD(r,d)^{\C^*}$ is a $\C^*$-fixed polystable Higgs bundle, we can write \begin{equation}\label{decomposition-polystable}(E,\phi) =(E_1,\phi_1)\oplus \cdots \oplus (E_k,\phi_k)\end{equation} with each $(E_i,\phi_i) \in \MD^{st}(r_i,d_i)$, then, each $(E_i,\phi_i)$ is necessarily $\C^*$-fixed. Indeed, given the uniqueness of the decomposition \eqref{decomposition-polystable} up to permutation of the factors, we have a homomorphism $\gamma:\C^* \to \mathfrak{S}_k$ such that $t \cdot(E_i,\phi_i)=(E_{\gamma(t)},\phi_{\gamma(t)})$ and such a homomorphism from $\C^*$ to the symmetric group $\mathfrak{S}_k$ must be trivial, since $\C^*$ does not have any finite index subgroup.
\end{remark}
The $\C^*$-action and the real structure $\sigma$ of $\MD(r,d)$ satisfy the following compatibility relation
\begin{equation}\label{anticommutative}
\sigma \big(t \cdot (E,\phi)\big)=\ov{t} \cdot \sigma\big((E,\phi)\big).
\end{equation} 
The $\C^*$-action on $\MD(r,d)$ therefore restricts to an $\R^*$-action on $\R \MD(r,d)$. The fixed-point set of this $\R^*$-action is described by the following result, which says that if a real point of $\MD(r,d)$ is fixed by $\R^*$, it is in fact fixed by the whole of $\C^*$.

\begin{prop}\label{Rstar_fixed_equiv_Cstar_fixed}
We have $(\R\MD(r,d))^{\R^*}=\R\MD(r,d) \cap \MD(r,d)^{\C^*}$.
\end{prop}

\begin{proof}
We need only show that $\R \MD(r,d)^{\R^*} \subseteq \R \MD(r,d) \cap \MD(r,d)^{\C^*}$, which follows from \cref{fixedpointsHodge} since not all elements of $\R^*$ are roots of unity in $\C^*$.
\end{proof}
As the definition of the space $\R\MD(r,d,s) \coloneqq c^{-1}(s)$ given in \cref{description-fibers} does not involve the Higgs field of a Higgs bundle $(E, \phi)$, the $\R^*$-action on $\R\MD(r,d)$ preserves the open-closed decomposition (\ref{fibers-topologicaltype}). More precisely, we have the following result.
\begin{prop}\label{limit-proposition}
For all $s \in \R\Pic_d(X)$, the BB flow $f: \C \to \MD(r,d)$ of a Galois-invariant Higgs bundle $(E,\phi) \in \R\MD(r,d,s)$ is a real BB flow that satisfies $f(\R) \subset \R\MD(r,d,s)$. In particular, the Higgs bundle $(E_0,\phi_0) \coloneqq f(0) $ satisfies $(E_0,\phi_0) \in \R\MD(r,d,s)$.
\end{prop}

\begin{proof}
By the compatibility relation of \cref{anticommutative} and the continuity of $\sigma$ and $f$, we have for all $t\in\C$, $\sigma(f(t)) = \si(t \cdot (E,\phi)) = \ov{t} \cdot \sigma (E,\phi) = \ov{t}  \cdot (E,\phi) = f(\ov{t})$. So $f$ is a real BB flow, which implies that $f(\R) \subset \R\MD(r,d)$. Since $\R$ is connected and $\R\MD(r,d,s)$ is open-closed by \cref{fibers-topologicaltype}, the fact that $f(1) = (E,\phi) \in \R\MD(r,d,s)$ implies that $f(\R) \subset \MD(r,d,s)$. In particular, $f(0) \in \R\MD(r,d,s)$, too.
\end{proof}

\begin{definition}\label{realHodgebundledefi}
Say that a real (respectively, quaternionic) Higgs bundle $((E,\phi),\tau)$ admits the structure of a real (respectively, quaternionic) system of Hodge bundles if the Higgs bundle $(E,\phi)$ admits the structure of a system of Hodge bundles such that, for all $p$ and $q$, we have $\tau(E^{p,q})=E^{p,q}$ and $\phi: E^{p,q} \to E^{p-1,q+1}\otimes \Omega^1_X$ is a morphism of real/quaternionic bundles.  
\end{definition}

We then have the following description of $\R^*$-fixed points in $\R\MD^{st}(r,d)$. Our proof follows \cite[Chapter 11]{Simpson_IHES_2}, adapting Simpson's argument to the real case.

\begin{lem}
\label{realHodgebundletheorem}
Let $(E,\phi,\tau)$ be a real (resp. quaternionic) Higgs bundle and $s\in \R^*\setminus\{1,-1\}$ such that $(E,\phi,\tau) \simeq (E,s \phi,\tau)$, then $(E,\phi)$ admits the structure of a real (resp. quaternionic) system of Hodge bundles.
\end{lem}

\begin{proof}
Let $f:(E,\phi,\tau) \to (E,s\phi,\tau)$ be an isomorphism of real/ quaternionic Higgs bundles. Then, in particular,
\begin{equation}\label{commutation-action}
    (f \otimes \mathrm{Id}_{\Omega^1_X}) \phi = s\phi f.
\end{equation} 
The coefficients of the characteristic polynomial $p_{f}(t)$ of $f$ are holomorphic functions over $X$ and are therefore constant. Let $r \coloneqq \mathrm{rk}(E)$ and denote by $E(f)$ the set of roots of $p_f(t)$. For all $\lambda \in \C$, let 
\begin{equation}\label{char_subbdle}
    E_{\lambda} \coloneqq \ker(f-\lambda \mathrm{Id}_E)^{r}
\end{equation} 
be the characteristic sub-bundle of $f$ associated to $\lambda$. Notice that $E_{\lambda} \neq 0$ if and only if $\lambda \in E(f)$. There is therefore a decomposition $E=\oplus_{\lambda \in E(f)}E_{\lambda}$ where $E_{\lambda}=\ker(f-\lambda \mathrm{Id}_E)^{r}$ is the characteristic sub-bundle of $f$ associated to the eigenvalue $\lambda$. We deduce from Equation (\ref{commutation-action}) that
$$
\big((f - s\lambda \mathrm{Id}_E) \otimes \mathrm{Id}_{\Omega^1_X}) \phi 
= (f \otimes \mathrm{Id}_{\Omega^1_X})\phi - s \lambda \phi
= s\phi f - s \lambda \phi
= s\phi(f-\lambda \mathrm{Id}_E)
$$
and by induction that, for all $m\in \mathbb{N}$,
$((f- s\lambda\mathrm{Id}_E)^m \otimes \mathrm{Id}_{\Omega^1_X}) \phi = s^m \phi(f-\lambda \mathrm{Id}_E)^m$. In particular, we have $\phi(\ker(f-\lambda \mathrm{Id}_E)^r) \subseteq \ker(f-s \lambda \mathrm{Id}_E)^r \otimes \Omega_X^1$. So either $s \lambda \not \in E(f)$ and $\phi(E_{\lambda})=0$ or $s \lambda \in E(f)$ and $\phi(E_{\lambda}) \subseteq E_{s \lambda} \otimes \Omega^1_X$. Now, since $f \circ \tau = \tau \circ f$ and $\tau$ is fibrewise $\C$-antilinear, we have $p_f(t) \in \R[t]$ and, in particular, if $\lambda \in E(f)$, then $\ov{\lambda} \in E(f)$ too. Moreover, $\tau(E_{\lambda})=E_{\ov{\lambda}}$. Since $s \in \R^*\setminus\{1,-1\}$, we have, for all $j>0$, $s^j \lambda \neq \overline{\lambda}$. Therefore, there exist distinct complex numbers $\lambda_1,\dots,\lambda_h$ and natural numbers $n_1,\ \ldots\ ,n_h \in \mathbb{N}$ such that the set of eigenvalues of $f$ can be written as $E(f)=I_1 \sqcup \ldots \sqcup I_k$
where by $I_j$ we mean the set
$I_{j} \coloneqq \{\lambda_j,s\lambda_j\,,\ldots\,,s^{n_j}\lambda_j,\overline{\lambda}_j,s\lambda_j,\ \ldots\ ,s^{n_j}\overline{\lambda_j}\}$
with the property that, for all $j$, $s^{-1}\lambda_j \not\in E(f)$ and $s^{n_j+1}\lambda_j \not \in E(f)$.  Moreover, for all $q \geq 0$, we have 
$\phi(E_{s^q\lambda_j}) \subseteq E_{s^{q+1}\lambda_j} \otimes \Omega^1_X$. Define then $E^{p,q}$ in the following way
\begin{equation}
\label{realSystemOfHodgeBundlesDef}
E^{p,q}
=
\left\{
\begin{array}{ll}
E_{s^q\lambda_{p+q+1}} & \text{if}\ \lambda_{p+q+1} \in \R, \\ E_{s^q\lambda_{p+q+1}} \oplus E_{s^q \overline{\lambda_{p+q+1}}} &  \text{if}\ \lambda_{p+q+1} \in \C \setminus \R.
\end{array}
\right.
\end{equation}
We see that $\tau(E^{p,q})=E^{p,q}$ for all $p,q$, so $(E,\phi)$ indeed admits the structure of a system of real/quaternionic Hodge bundles in the sense of \cref{realHodgebundledefi}.
\end{proof}

As an application of \cref{realHodgebundletheorem}, we get the following result.

\begin{corollary}
Let $(E,\phi)\in \R\MD^{st}(r,d)^{\R^*}$ and let $\tau$ be the real/quaternionic structure of $\tau$ given by \cref{realHiggsbundlemoduli}. Then $(E,\phi,\tau)$ admits the structure of a real/quat\-ernionic system of Hodge bundles. Moreover, this structure is unique up to permutation of the indices.
\end{corollary}

In fact, in the presence of a given real/quaternionic structure, this result can be generalized as follows. Consider a stable but not necessarily geometrically stable real or quaternionic Higgs bundle $(E,\phi,\tau)$ and assume that it is a fixed point of the $\R^*$-action on $\R\MD^{st}(r,d))$. Using the notations introduced in the proof of \cref{realHodgebundletheorem}, we put, for all $j=1,\dots,h$, $E_j\coloneqq \oplus_{\lambda \in I_j} E_{\lambda}$.
By construction, each $E_j$ is $\phi$-invariant, i.e.\ $\phi(E_j) \subseteq E_j \otimes \Omega^1_{X}$, as well as $\tau$-invariant, i.e.\ $\tau(E_j) = E_j$. We thus have a decomposition 
$
(E,\phi,\tau)
=
(E_1,\phi|_{E_1}, \tau|_{E_1}) \oplus \cdots \oplus (E_h,\phi|_{E_h},\tau|_{E_h}) 
$.
Since, given two vector bundles $F,F'$ on $X$, we have 
$\mu(F \oplus F') \leq  \max(\mu(F),\mu(F'))$,
we deduce that $\mu(E) \leq \max_{i=1}^h \mu(E_i)$. If $h \neq 1$, then the fact that $(E,\phi,\tau)$ is stable in the real/quaternionic sense implies that, for all $j\in \{1, \ldots, h\}$, $\mu(E_j) < \mu(E)$. So $\max_{j=1}^h \mu(E_j) < \mu(E)$ and we get a contradiction. Hence $h=1$. Then, if we put $n \coloneqq \mathrm{card}\ I_1$ and, for all $p\geq 0$, $E^{n-p}\coloneqq E^{-p,p}$, where $E^{p,q}$ is defined as in \cref{realSystemOfHodgeBundlesDef}, we obtain a decomposition 
\begin{equation}\label{decompositionfixedrealpoints}
E=E^n \oplus \cdots \oplus E^0
\end{equation}
such that, for all $p$, $\tau(E^p) = E^p$ and $\phi(E^p) \subseteq E^{p-1} \otimes \Omega^1_{X}$. So the structure of a stable real/quaternionic Higgs bundle $(E,\phi,\tau)$ that is fixed by the $\R^*$-action is a close analogue of that of a stable Higgs bundle that is fixed by the $\C^*$-action, even when $(E,\phi,\tau)$ is not geometrically stable. More precisely, using the definition of the $E^{p,q}$ given in \cref{realSystemOfHodgeBundlesDef} depending on whether the characteristic value $\lambda\in\C$ is real or not, we obtain the following statement.

\begin{theorem}\label{stableRealQuatHodgeBundles}
    Let $(E,\phi,\tau)$ be a stable real/quaternionic Hodge bundle and assume there is an $s \in \R^*\setminus\{1,-1\}$ such that $s \cdot (E,\phi,\tau) \simeq (E,\phi,\tau)$.
    \begin{enumerate}
        \item If $(E,\phi,\tau)$ is geometrically stable, then there exist $\lambda\in\R$ and $n\in\mathbb{N}$ such that 
        $
        E = E_\lambda \oplus E_{s\lambda} \oplus \, \ldots \, \oplus E_{s^n \lambda}
        $
        and, for all $j\in\{0,\ \ldots\ ,n\}$, $\tau(E_{s^j\lambda})= E_{s^j\lambda}$ and $\phi(E_{s^j\lambda}) \subset E_{s^{j+1}\lambda} \otimes \Omega^1_X$, where for all $\mu\in\C$, the characteristic sub-bundle $E_\mu$ is defined as in \cref{char_subbdle}.
        \item If $(E,\phi,\tau)$ is stable but not geometrically stable, then there exist $\lambda \in \C \setminus \R$ and $n\in\mathbb{N}$ such that 
        $
        E = (E_\lambda \oplus E_{\ov{\lambda}})\oplus (E_{s\lambda} \oplus E_{s\ov{\lambda}}) \oplus \, \ldots \, \oplus (E_{s^n \lambda} \oplus E_{s^n\ov{\lambda}})
        $
        and, for all $j\in\{0,\ \ldots\ ,n\}$, $\tau(x,y) = (\pm y,x)$ on $E_{s^j\lambda}\oplus E_{s^j\ov{\lambda}}$, where the $\pm$ sign depends on whether $\tau$ is real or quaternionic, and $\phi(E_{s^j\lambda}\oplus E_{s^j\ov{\lambda}}) = (E_{s^{j+1}\lambda}\oplus E_{s^{j+1}\ov{\lambda}})\otimes\Omega^1_X$.
    \end{enumerate}
\end{theorem}

Theorem \ref{stableRealQuatHodgeBundles} gives an explicit description of the structure of stable real and quaternionic Hodge bundles, depending on whether they are geometrically stable or not. However, in practice, the more abstract description  given in \cref{decompositionfixedrealpoints} is often sufficient (see for instance the proof of \cref{downwardflowsfixed}).

\subsection{Connected components of the real locus}\label{CC_real_locus_section}

We give in \cref{downwardflowsfixed} a real version of \cite[Lemma~11.9]{Simpson_IHES_2}. In Simpson's work, the latter result is used to show (in \cite[Corollary~11.10]{Simpson_IHES_2}) that it is possible to connect any two points $(E_1,\phi_1)$ and $(E_2,\phi_2)$ in $\MD(r,d)$, by connecting each one of them to a point in $\mathcal{N}(r,d)$ via a series of BB flows and exploiting the connectedness of $\mathcal{N}(r,d)$, as illustrated in \cref{fig:BB_flows_btw_pts}.

\begin{figure}[!ht]
    \centering
   \begin{tikzpicture}[scale=1.2]
    \fill[blue!30] (-1,0) -- (-3,4) to (3,4) -- (1,0) -- cycle;
    
    \draw[very thick] (-1.75,1.5) -- (-1.25,0.5);
    \draw[very thick] (-2,2) -- (-3,4);
    \draw[very thick] (2,2) -- (3,4);
    \draw[very thick] (1.75,1.5) -- (1.25,0.5);
    \draw[very thick] (-1,0) -- (1,0);
    
    \node at (-2.3,0.6) {$\mathbf{M}_{\mathrm{Dol}}(r,d)^{\mathbb{C}^*}$};
    \node at (0,3) {Nilpotent cone};
    \node at (-1.5,-0.25) {$\mathcal{N}(r,d)$};
    
    \node[circle,fill,inner sep=1.5pt] at (-4.1,3.525) {};
    \node at (-4.7,3.5) {$(E_1,\phi_1)$};
            
    \node[circle,fill,inner sep=1.5pt] at (-2.45,2.9) {};
    
    \draw[-Triangle, thick, line join=round,
decorate, decoration={
    zigzag,
    segment length=8,
    amplitude=2,post=lineto,
    post length=2pt
}] 
        (-4.1,3.5) to (-2.6,3);
    
    \node[circle,fill,inner sep=1.5pt] at (-1,2.05) {};

    \node[circle,fill,inner sep=1.5pt] at (-1.6,1.2) {};

    \node[circle,fill,inner sep=1.5pt] at (-0.45,1) {};
    
    \draw[-Triangle, thick, line join=round,
decorate, decoration={
    zigzag,
    segment length=8,
    amplitude=2,post=lineto,
    post length=2pt
}] 
        (-1,2.05) to (-2.28,2.8);

    \node[circle,fill,inner sep=1.5pt] at (2.5,1) {};
            
    \node[circle,fill,inner sep=1.5pt] at (1.45,0.9) {};
    \node at (3.2,1) {$(E_2,\phi_2)$};
    
        \node[circle,fill,inner sep=1.5pt] at (-0.3,0) {};

    \draw[-Triangle, thick, line join=round,
decorate, decoration={
    zigzag,
    segment length=8,
    amplitude=2,post=lineto,
    post length=2pt
}] 
        (2.5,1) to (1.6,0.9);
    
    \draw[-Triangle, thick, line join=round,
decorate, decoration={
    zigzag,
    segment length=8,
    amplitude=2,post=lineto,
    post length=2pt
}] 
        (-1,2) to (-1.55,1.3);

    \draw[-Triangle, thick, line join=round,
decorate, decoration={
    zigzag,
    segment length=8,
    amplitude=2,post=lineto,
    post length=2pt
}] 
         (-0.5,1) to (-1.55,1.2);

    \draw[-Triangle, thick, line join=round,
decorate, decoration={
    zigzag,
    segment length=8,
    amplitude=2,post=lineto,
    post length=2pt
}] 
        (-0.5,1) to (-0.325,0.08);
        
    \node[circle,fill,inner sep=1.5pt] at (0.7,0.7) {};
    
    \draw[-Triangle, thick, line join=round,
decorate, decoration={
    zigzag,
    segment length=8,
    amplitude=2,post=lineto,
    post length=2pt
}] 
        (0.7,0.65) to (1.35,0.9) ;

    \node[circle,fill,inner sep=1.5pt] at (0.25,0) {};
    
    \draw[-Triangle, thick, line join=round,
decorate, decoration={
    zigzag,
    segment length=8,
    amplitude=2,post=lineto,
    post length=2pt
}] 
        (0.65,0.65) to (0.35,0.1);

\end{tikzpicture}
    \caption{Connecting any two points via Bialynicki-Birula flows.}
    \label{fig:BB_flows_btw_pts}
\end{figure}
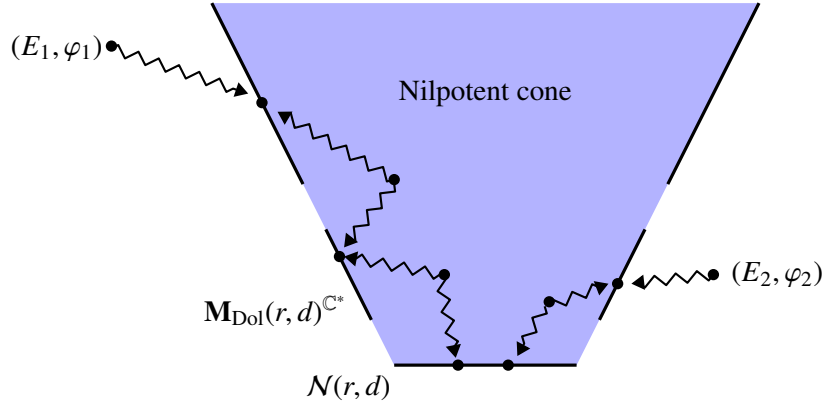

\smallskip

In the real setting, we will show in \cref{CC_real_locus} that, although $\R\mathcal{N}(r,d)$ is not connected in general, we can refine Simpson's arguments to count the connected components of $\R\MD(r,d)$, at least when $\gcd(r,d)=1$. This is achieved by considering the real part of the Bialynicki-Birula flows, whose existence is guaranteed by the anti-commutation property of \cref{anticommutative}, and by using the fact that, when $\gcd(r,d) = 1$, we can count the connected components of $\R\mathcal{N}(r,d)$. The first step towards this is \cref{downwardflowsfixed}. Recall that, for all $s \in \pi_0(\R\Pic_d(X))$, we denote by $\R\MD(r,d,s)$ the fibre $c^{-1}(s)$ of the map $c : \R\MD(r,d) \to \pi_0(\R\Pic_d(X))$ introduced in \cref{obstruction_map}.

\begin{lem}
\label{downwardflowsfixed}
Let $s \in \pi_0(\R\Pic_d(X))$ and let $(E,\phi,\tau)$ be a stable real/quaternionic Higgs bundle such that $(E,\phi)\in\R\MD(r,d,s)^{\R^*}$. If $\phi\neq 0$, then there exists a stable real/quaternionic Higgs bundle $(F,\theta,\tau_F)$ such that $(F,\theta) \not\simeq (E,\phi)$, and a complete real BB flow $f : \R\mathbf{P}^1 \to \R\MD(r,d,s)$ such that $f(1) = (F,\theta)$ and $f(\infty) = (E,\phi)$. In particular, $f(0) = \lim_{t \to 0} t \cdot (F,\theta)$ also lies in $\R\MD(r,d,s)$.
\end{lem}

\begin{proof}[Proof of \cref{downwardflowsfixed}]
We distinguish two cases, depending on whether the real or quaternionic Higgs bundle $(E,\phi, \tau)$ is geometrically stable or not. 

\smallskip

\textbf{Case 1.} Assume that $(E,\phi,\tau)$ is geometrically stable. Using the same notation as in \cref{decompositionfixedrealpoints}, we can write $E = E^n \oplus\, \ldots\, \oplus E^0$ with $\phi(E^j) \subset E^{j-1} \otimes \Omega^1_X$ and $\tau(E^j) = E_j$. Note that $n > 0$  (otherwise $\phi=0$, which contradicts our assumption on $\phi$) and that, for all $j$ the vector bundle $E'_j \coloneqq E^j \oplus \, \ldots\, \oplus E^0$ is $\phi$-invariant. As the Higgs bundle $(E,\phi)$ is stable, we then have $\mu(E'_{n-1}) < \mu(E)$ and this implies that $\mu(E^n) = \mu(E/E'_{n-1}) > \mu(E)$. Since $E^0$ is $\phi$-invariant, we also have $\mu(E^0) < \mu(E)$. Thus, $\mu(E^0) < \mu(E) < \mu(E^n)$. Using these two inequalities, we deduce that
\begin{equation}\label{dimensionhom}
\begin{array}{rcl}
\deg(E^0 \otimes (E^n)^*) & = & \deg(E^0)\rk(E^n)-\deg(E^n)\rk(E^0) \\
& = & \rk(E^0)\rk(E^n)(\mu(E^0)-\mu(E^n)\big) \\
& < & \rk(E^0)\rk(E^n)(\mu(E)-\mu(E^n)\big) \\
& < & 0.
\end{array}
\end{equation}
Applying the Riemann-Roch formula
$h^0(X,\mathcal{E}) - h^1(X,\mathcal{E}) = \deg(\mathcal{E}) + \mathrm{rk}(\mathcal{E})(1-g)$
to the vector bundle $\mathcal{E} = E^0 \otimes (E^n)^*$, we deduce from \ref{dimensionhom} that $h^1(X,\mathcal{E}) >0$, hence
\begin{equation}
\label{Ext_nonzero}
\Ext^1\big(E^n,E^0\big) = H^1\big(X,E^0 \otimes (E^n)^*\big) \neq \{0\}.    
\end{equation}
Note now that, whether $\tau$ be real or quaternionic, the morphism $\tau|_{E_0} \otimes (\tau|_{E_n})^*$ endows the vector bundle $E^0 \otimes (E^n)^*$ with a real structure, since $(\tau \otimes \tau)^2 = \tau^2 \otimes \tau^2 = (\pm I_{E^0}) \otimes (\pm I_{(E^n)^*}) =I_{E^0 \otimes (E^n)^*}$. The complex vector space $H^1(X,E^0 \otimes (E^n)^*)=\Ext^1(E^n,E^0)$ therefore has an induced real structure, that we still denote by $\tau:\Ext^1(E^n,E^0) \to \Ext^1(E^n,E^0)$. At the level of extensions, this real structure has the following concrete interpretation. Denote by $\alpha_n:E^n \to \sigma^*(\overline{E^n}) $ and $\alpha_0:E^0 \to \sigma^*(\overline{E^0})$ the isomorphisms of complex vector bundles given by the real/quaternionic structures of $E^n$ and $E^0$. Given a short exact sequence 
$$
\eta = \big(0 \rightarrow E^0 \xrightarrow{f} M \xrightarrow{g} E^n \rightarrow 0\big) \in \Ext^1(E^n,E^0),
$$ 
its image by $\tau$ is given by 
$$
\tau(\eta)= \big(0 \rightarrow E^0 \xrightarrow{\sigma^*(\overline{f}) \circ \alpha_0} \sigma^*(\overline{M}) \xrightarrow{\alpha_n^{-1}\circ \sigma^*(\overline{g})} E^n \rightarrow 0\big).
$$
Pick now a non-zero real short exact sequence, i.e.\ a non-trivial element $\eta \in \Ext^1(E^n,E^0)^{\tau}$, which is non-zero by \cref{Ext_nonzero}, since $\tau$ is a linear real structure. As in Simpson's proof of \cite[Lemma~11.9]{Simpson_IHES_2}, consider, for all $t\in \R$, the real extension 
$$
t^n \eta = \big(0 \rightarrow E^0 \xrightarrow{f_t} M_t \xrightarrow{g_t} E^n \rightarrow 0\big).
$$
This comes equipped with an isomorphism of complex vector bundles
\begin{equation}
\label{isomorhismM}
\alpha_t : M_t \to \sigma^*(\overline{M_t})\end{equation} such that $\sigma^*(\overline{\alpha_t}) \alpha_t = \pm \mathrm{Id}_M$, depending on whether $(E,\phi,\tau)$ is a real/quaternionic Higgs bundle. Let then $(F_t,\theta_t)$ be the Higgs bundle given, for all $t\in \C$, by
$$F_t = M_t \oplus E^{n-1} \oplus\, \ldots\, \oplus E^1 $$
where $\theta_t$ is given by $\phi$ on the summands $E^j$ for $2 \leq j \leq n-1$, by 
$$ 
E^1 \xrightarrow{\phi} E^0 \otimes \Omega^1_X \xrightarrow{f_t \otimes \mathrm{Id}_{\Omega^1_X}} M \otimes \Omega^1_X
$$ on the summand $E^1$ and by 
$$
M_t \xrightarrow{g_t} E^n \xrightarrow{\phi} E^{n-1} \otimes \Omega^1_X
$$ 
on the summand $M_t$. The argument in Simpson's proof shows that $(F_t,\theta_t)$ is stable. We deduce from \cref{realHiggsbundlemoduli} and \cref{isomorhismM} that, if $t\in\R$, then $(F_t,\theta_t)$ has an induced real/quaternionic structure, with respect to which it is therefore geometrically stable. So, $\forall t\in\R, (F_t,\theta_t) \in \R\MD^{st}(r,d)$. Simpson's proof also shows that  $(F,\theta) \coloneqq (F_1,\theta_1) \neq (E,\phi)$ and that $\lim_{t\to\infty}t\cdot(F,\theta) = (E,\phi)$. So, the map $t \mapsto t \cdot (F,\theta)$ defines a complete real BB flow $f : \R\mathbf{P}^1 \to \R\MD(r,d)$ such that $f(\infty) = (E,\phi)\in\R\MD(r,d,s)$. By connectedness of $\R\mathbf{P}^1$ and open-closedness of $\R\MD(r,d,s)$, we have $f(\R\mathbf{P}^1)\subset \R\MD(r,d,s)$. In particular, $f(1) = (F,\theta)$ and $f(0) = \lim_{t\to 0} t\cdot(F,\theta)$ also lie in $\R\MD(r,d,s)$.

\smallskip

\textbf{Case 2.} Assume now that $(E,\phi, \tau)$ is stable but not geometrically stable. By Remark \ref{realstable-not-geom}, the rank $r$ and degree $d$ of $E$ are necessarily even and there exists a stable Higgs bundle $(\mathcal{E},\Phi) \in \MD^{st}(r/2,d/2)$ such that $\si(\mathcal{E},\Phi) \not\simeq (\mathcal{E},\Phi)$ and
$
(E,\phi) \simeq (\mathcal{E},\Phi) \oplus \si(\mathcal{E},\Phi)
$.
By \cite[Lemma 11.9]{Simpson_IHES_2}, there exists a complete BB flow $g: \mathbf{P}^1 \to \MD(r/2,d/2)$ such that $(\mathcal{F},\Theta) \coloneqq g(1) \not\simeq (\mathcal{E},\Phi)$ and $g(\infty) = (\mathcal{E},\Phi)$. In particular, $g$ is noncontant and $(\mathcal{F},\Theta)$ is not $\C^*$-fixed. Note that $(\mathcal{F},\Theta) \not \simeq \si(\mathcal{F},\Theta) $, as otherwise we would have
$$
(\mathcal{E},\Phi) = \lim_{t\to\infty} t\cdot(\mathcal{F},\Theta) = \lim_{t\to\infty} t\cdot\si(\mathcal{F},\Theta) = \si(\lim_{t\to\infty} \ov{t} \cdot (\mathcal{F},\Theta)\big) = \sigma(\mathcal{E},\Phi),
$$
contradicting our assumption on $(\mathcal{E},\Phi)$. Consider then the complete BB flow
$$
\begin{array}{rcl}
f : \mathbf{P}^1 & \lra & \MD(r,d) \\
t & \lmt & g\big(t\big) \oplus \si\big(g(\ov{t})\big)
\end{array}
$$
and set $(F,\theta) \coloneqq f(1)$. Since $g(\ov{t}) = \ov{t}\cdot (\mathcal{F},\Theta)$, by \cref{anticommutative} we have $\si(g(\ov{t})) = t \cdot \si(\mathcal{F},\Theta)$. So, $\forall t\in\C, f(t) = t\cdot ((\mathcal{F},\Theta)\oplus \si(\mathcal{F},\Theta))$. In particular, $(F,\theta) = (\mathcal{F},\Theta) \oplus \si(\mathcal{F},\Theta) \in \R\MD(r,d)$. This shows that $(F,\theta)$ is not fixed by the $\C^*$-action since, if it was, then by \cref{stable-factors} $(\mathcal{F},\Theta)$ would also be fixed by the $\C^*$-action. In particular, $(F,\theta) \neq (E,\phi)$. Moreover, 
$$
f(\infty) = g(\infty) \oplus \si\big(g(\infty)\big) = (\mathcal{E},\Phi) \oplus \si(\mathcal{E},\Phi) = (E,\phi) \in \R\MD(r,d,s).
$$
Therefore, as in the previous case, $f(\R\mathbf{P}^1) \subset \R\MD(r,d,s)$. In particular, $f(1) = (F,\theta)$ and $f(0) = \lim_{t\to 0} t\cdot(F,\theta)$ also lie in $\R\MD(r,d,s)$.
\end{proof}

As a consequence of \cref{downwardflowsfixed}, we obtain the following result, the complex analogue of which appears within the proof of \cite[Corollary 11.10]{Simpson_IHES_2}.

\begin{prop}
\label{flowrealpart-ss}
Let $s \in \pi_0(\R\Pic_d(X))$ and let $(E,\phi) \in \R\MD(r,d,s)^{\R^*}$ such that $\phi \neq 0$. Then there exists $(F,\theta) \in \R\MD(r,d)$ such that $(F,\theta) \not\simeq (E,\phi)$ and a complete real BB flow $f : \R\mathbf{P}^1 \to \R\MD(r,d,s)$ such that $f(1) = (F,\theta)$ and $f(\infty) = (E,\phi)$. In particular $f(0) = \lim_{t\to 0} t \cdot (F,\theta)$ also lies in $\R\MD(r,d,s)$.
\end{prop}

\begin{proof}
Since $(E,\phi)$ is a Higgs bundle which is both polystable and Galois-invariant, Proposition \ref{polystable-real} shows that there exist stable (but not necessarily geometrically stable) real/quaternionic Higgs bundles $(E_1,\phi_1,\tau_1),\dots,(E_\ell,\phi_\ell,\tau_\ell)$, all of them of slope $d/r$, such that 
$
(E,\phi) \simeq (E_1,\phi_1) \oplus \cdots \oplus (E_\ell,\phi_\ell)
$
as Higgs bundles. Since $(E,\phi)$ is $\R^*$-fixed, \cref{Rstar_fixed_equiv_Cstar_fixed} shows that it is also $\C^*$-fixed. So \cref{stable-factors} shows that each $(E_i,\phi_i)$ is $\C^*$-fixed. Since $\phi \neq 0$, there exists $i$ such that $\phi_i \neq 0$. Up to reordering, we can assume that $\phi_1 \neq 0$. We can thus write $(E,\phi)=(E_1,\phi_1)\oplus (E_1',\phi_1')$ where $(E_1,\phi_1)$ comes equipped with a stable real/quaternionic structure $\tau_1$. Moreover, $(E_1,\phi_1,\tau_1)$ is $\R^*$-fixed and, by construction, $\phi_1 \neq 0$. Therefore, by \cref{downwardflowsfixed}, there exists a stable real/quaternionic Higgs bundle $(F_1,\theta_1,\tau_{F_1})$ such that $(F_1,\theta_1) \not\simeq (E_1,\phi_1)$, and a complete real BB flow $g : \R\mathbf{P}^1 \to \R\MD(r_1,d_1)$ such that $g(1) = (F_1,\theta_1)$ and $\lim_{t\to\infty} t \cdot (F_1,\theta_1)= (E_1,\phi_1)$. We then put, for all $t\in\R^*$, 
$$
(F_t,\theta_t) \coloneqq \big(t\cdot (F_1,\theta_1)\big) \oplus (E_1',\phi_1') = t \cdot \big((F_1,\theta_1) \oplus (E_1',\phi_1')\big),
$$
The map $f(t) \coloneqq (F_t,\theta_t)$ defines a complete real BB flow $f : \R\mathbf{P}^1 \to \R\MD(r,d)$ such that $(F,\theta) \coloneqq f(1) \not\simeq (E,\phi)$ and $f(\infty) = (E,\phi)\in \R\MD(r,d,s)$. Therefore, as in \cref{downwardflowsfixed}, $f(\R\mathbf{P}^1) \subset \R\MD(r,d,s)$ and in particular $f(0) \in \R\MD(r,d,s)$.
\end{proof}

We can now show our main result, which is \cref{CC_real_locus}. We need the following notation, introduced by Simpson in \cite[Chapter~11, pp.~75-77]{Simpson_IHES_2}. Recall that the action of $\C^*$ on $\MD(r,d)$ can be linearized. In particular, there exists a vector space $V$ with a linear $\C^*$-action such that $\MD(r,d)$ is an open subset of a $\C^*$-stable projective subvariety $Z \subseteq \mathbb{P}(V)$ and the action of $\C^*$ is induced by this inclusion. Consider the decomposition of $V$ into weight spaces $V=\bigoplus_{\alpha \in W} V_{\alpha}$, where $V_\alpha \coloneqq \{ z \in V\ |\ \forall t\in \C^*,\ t \cdot z = t^\alpha z \}$ and $W \coloneqq \{\alpha \in \Z\ |\ V_\alpha \neq 0\}$ is the set of weights, which is finite since $\dim V<\infty$. Setting $Z_{\alpha} \coloneqq Z \cap \mathbb{P}(V_{\alpha})$, we have open-closed decompositions 
$$
\mathbb{P}(V)^{\C^*}=\bigsqcup_{\alpha \in W} \mathbb{P}(V_{\alpha})\ ,
\quad
Z^{\C^*}=\bigsqcup_{\alpha \in W} Z_{\alpha}\ ,\quad \mathrm{and} \quad
\MD(r,d)^{\C^*}=\bigsqcup_{\alpha \in W} \MD(r,d)_{\alpha}\ ,
$$ where $\MD(r,d)_{\alpha} \coloneqq Z_{\alpha} \cap \MD(r,d)$. In particular, there is a well-defined function $\beta: Z^{\C^*} \to \Z$ sending $p \in Z^{\C^*}$ to the unique $\alpha\in W$ such that $z\in Z_\alpha$ (so, by definition of $\beta$, we have $p \in Z_{\beta(p)}$ for all $p\in Z^{\C^*}$). Since $\beta$ is continuous, it is constant on every $F \in \pi_0(\MD(r,d)^{\C^*})$. Moreover, since $Z$ is compact, we have that, for all $z \in Z$, the limits $\lim_{t \to 0} t \cdot z$ and $\lim_{t \to \infty} t \cdot z \in Z^{\C^*}$ exist. Denote by $\beta^0(z),\beta^{\infty}(z) \in W$ the elements such that $\lim_{t \to 0} t \cdot z \in Z_{\beta^0(z)}$ and $\lim_{t \to \infty} t \cdot z \in Z_{\beta^{\infty}(z)}$. Simpson shows that, for all $z \in Z$, $\beta^0(z) \leq \beta^{\infty}(z)$ and that $\beta^0(z)=\beta^{\infty}(z)$ if and only if $z$ is $\C^*$-fixed (\cite[p.75]{Simpson_IHES_2}). In the case of $\MD(r,d)$, up to rescaling, we can assume that $\beta(\mathcal{N}(r,d)) = 0$. Notice that \cite[Lemma~11.8 and Corollary~11.10]{Simpson_IHES_2} imply that for every $p \in \MD(r,d)^{\C^*}\setminus \mathcal{N}(r,d)$, we have $\beta(p) >0$. Indeed, assume that there exists $p\in\MD(r,d)^{\C^*}\setminus \mathcal{N}(r,d)$ such that $\beta(p) \leq 0$ and take such a $p$ corresponding to the minimal weight, so that $\beta(p) \leq \beta(p')$ for all $p' \in \MD(r,d)^{\C^*}$. The arguments of \cite[Corollary~11.10]{Simpson_IHES_2} show that there exists $z \in \MD(r,d)\setminus \{p\}$ such that $\lim_{t \to \infty} t \cdot z = p$. In particular, $z$ is not a $\C^*$-fixed point and thus $\beta(\lim_{t \to 0} t \cdot z)=\beta^0(z) < \beta^{\infty}(z)=\beta(\lim_{t \to \infty} t \cdot z) = \beta(p)$, which contradicts the minimality of $\beta(p)$. We now apply these ideas in the real context.

\begin{theorem}
\label{implication-realpart}
For all $s \in \pi_0(\R\Pic_d(X))$, if $\R\mathcal{N}(r,d,s)$ is connected, then so is $\R\MD(r,d,s)$.
\end{theorem}

\begin{proof}
Fix $s \in \pi_0(\R \Pic_d(X))$ and consider a point $(E,\phi) \in \R\MD(r,d,s)$. By \cref{limit-proposition}, there is a real BB flow $f : \R \to \R\MD(r,d,s)$ such that $f(1) = (E,\phi)$ and the $\R^*$-fixed Higgs bundle $(E_0,\phi_0) \coloneqq f(0)$ is Galois-invariant and belongs to the same $\R\MD(r,d,s)$ as $(E,\phi)$. We may therefore assume that $(E,\phi) \in \R\MD(r,d,s)^{\R^*}$ to begin with. If $\phi = 0$, then $(E,\phi)$ is already in $\mathcal{N}(r,d,s)$ and we are done. If $\phi \neq 0$, then by \cref{flowrealpart-ss}, there exists a Galois-invariant Higgs bundle $(F,\theta)$ and a complete real BB flow $f: \R\mathbf{P}^1 \to \R\MD(r,d,s)$ such that $f(1) = (F,\theta) \not \simeq (E,\phi)$ and $f(\infty) = (E,\phi)$. In particular, the Higgs bundle $(F_0,\theta_0) \coloneqq f(0)$ also belongs to $\R\MD(r,d,s)$. Let $\beta \coloneqq \beta^0(E,\phi) = \beta^\infty(E,\phi)\in \Z$ be the weight such that $(E,\phi) \in \MD(r,d)_\beta$. Since $(E,\phi) = \lim_{t\to \infty}t\cdot(F,\theta)$, we have $\beta^\infty(F,\theta) = \beta$. And since $(F,\theta) \not \in \R\MD(r,d,s)^{\R^*}$, we have $\beta' \coloneqq \beta^0(F,\theta) < \beta^\infty(F,\theta) = \beta$. So $(F_0,\theta_0) \coloneqq \lim_{t\to 0}t\cdot(F,\theta) \in \MD(r,d)_{\beta'}$ with $\beta' < \beta$. This means that we have connected the Galois-invariant and $\R^*$-fixed Higgs bundle $(E,\phi)\in\R\MD(r,d,s)\cap\MD(r,d)_\beta$ to a Galois-invariant and $\R^*$-fixed Higgs bundle $(F_0,\theta_0) \in\R\MD(r,d,s)\cap\MD(r,d)_{\beta'}$ with $\beta' < \beta$, via a path which is entirely contained in $\R\MD(r,d,s)$. If $\theta_0 = 0$, then $(F_0,\theta_0) \in \R\mathcal{N}(r,d,s)$ and we are done. If $\theta_0 \neq 0$, then we repeat the procedure above to get to a weight $\beta'' < \beta'$. As the set of weights is finite, we eventually get to the minimal weight $\beta'' = 0$ and deduce that all points of $\R\MD(r,d,s)$ can be connected, via a path entirely contained in $\R\MD(r,d,s)$, to a point in $\R\mathcal{N}(r,d,s)$. So, if $\R\mathcal{N}(r,d,s)$ is connected, then so is $\R\MD(r,d,s)$.
\end{proof}

\begin{corollary}
\label{CC_real_locus}
Assume that $\gcd(r,d)=1$. Then, for all $s \in \pi_0(\R \Pic_d(X))$, the space $\R\MD(r,d,s)$ is connected. Therefore, the open-closed decomposition in \cref{fibers-topologicaltype} is the decomposition of $\R\MD(r,d)$ into connected components and the latter are indexed by topological types of Galois-invariant stable Higgs bundles.
\end{corollary}

\begin{proof}
By \cref{implication-realpart}, it suffices to prove that, for all $s$, the space $\R\mathcal{N}(r,d,s)$ is connected when $\gcd(r,d)=1$. But this is known to hold by \cite[Theorem~3.8 and Corollary~3.9]{Sch_JSG}) when $g\geq 2$ and by \cite[Theorem~1.2]{BiSc} when $g=1$. As a consequence, the open-closed decomposition in \cref{fibers-topologicaltype} coincides with the decomposition of $\R\MD(r,d)$ into connected components and, by the results of \cref{description-fibers}, the latter are indeed indexed by topological types of Galois-invariant stable Higgs bundles.
\end{proof}

We will see in \cref{example-elliptic-real} an example in which the assumption of \cref{implication-realpart} is satisfied without assuming that $\gcd(r,d) = 1$. In fact, it is natural to conjecture that $\R\mathcal{N}(r,d,s)$ is connected for all $r,d,s$. Recall indeed that, in the complex setting, we have a decomposition into connected components 
$
\Pic(X) 
=
\sqcup_{d \in \Z} \Pic_d(X)
$ 
and that, for all $d \in \Z$ and all $r >0$ the moduli space of vector bundles $\mathcal{N}(r,d)$ is connected. Likewise, in the real setting, we have a decomposition into connected components 
$$ 
\R \Pic(X) \qquad
= \quad
\bigsqcup_{d\in\Z,\ s\in\pi_0(\R\Pic_d(X))} \R \Pic_d(X)(s)
$$ 
and we expect the latter to give rise to connected spaces $\R\mathcal{N}(r,d,s)$ even when $\gcd(r,d) \neq 1$. The issue is that, when $\gcd(r,d) \neq 1$, we do not have a good modular interpretation of $\R\mathcal{N}(r,d,s)$  as moduli spaces of real/quaternionic vector bundles with fixed topological type, and this makes the arguments of \cite{BHH},\cite{Sch_JSG} insufficient to prove the connectedness of these spaces.

\subsection{Higgs bundles on real elliptic curves}
\label{example-elliptic-real}

Moduli spaces of Higgs bundles over a real elliptic curve $(X,\si)$ have been studied in \cite{BiCaFraGP}. We show in \cref{proposition-elliptic-curves} that, when $g=1$, we can use \cref{implication-realpart} to study connected components of $\R\MD(r,d)$ even when $\gcd(r,d) > 1$. For the sake of simplicity, we treat the case when $\gcd(r,d)=2$, but a similar reasoning can be applied in the general case.

\begin{prop}
\label{proposition-elliptic-curves}
Let $(X,\si)$ be a real curve of genus $g=1$. Then, for all $r,d$ such that $ \gcd(r,d)=2$ and all $s \in \pi_0(\R \Pic_d(X))$, the moduli space $\R \mathcal{N}(r,d,s)$ and thus $\R \MD(r,d,s)$ is connected.    
\end{prop}

For all $x,y \in X$, we put $\mathcal{O}_X(x+y)\coloneqq \mathcal{O}_X(x) \otimes \mathcal{O}_X(y)$ and $\mathcal{O}_X(-x) \coloneqq O_X(x)^{-1}$. Fix $x_0 \in X$ and denote by $+_{x_0}: X \to X$ the associated sum operation having $x_0$ as neutral element (the inverse operation will be denoted by $-_{x_0}$). Recall that this is defined so that, for all $x,y \in X$, we have 
 \begin{equation}
    \label{sum-elliptic}
    \mathcal{O}_X\big((x+_{x_0}y)-x_0\big) \simeq \mathcal{O}_X\big(x-x_0\big) \otimes \mathcal{O}_X\big(x-x_0\big).
\end{equation}
Let now $\sigma:X \to X$ be an anti-holomorphic involution. A description of the real locus $\R X$ can be found for instance in \cite{GH}. In particular, we can have $n=0,1$ or $2$. Let $e \coloneqq \gcd(r,d)$ and $d'\coloneqq d/e$. By \cite[Theorem 1]{Tu_elliptic}, there is an isomorphism of real algebraic varieties 
$
(\mathcal{N}(r,d),\sigma) \simeq (\Sym^e(\Pic_{d'}(X)),\sigma) .
$
Through this isomorphism, the determinant map $\mathcal{N}(r,d) \to \Pic_d(X)$ corresponds to the map 
\begin{equation}
\label{map-elliptic-real}
\begin{array}{rcl}
\theta:\Sym^e\big(\Pic_{d'}(X)\big) & \lra & \Pic_d(X)\\
\left[(\mathcal{L}_1,\dots,\mathcal{L}_e)\right] & \lmt & \left[\bigotimes_{i=1}^e \mathcal{L}_i\right]
\end{array}
\end{equation}
so, for all $s \in \pi_0(\R \Pic_d(X))$, $\R \mathcal{N}(r,d,s) = \theta^{-1}(s) \cap \R\Sym^e(\Pic_{d'}(X))$. 
\begin{proof}[Proof of Proposition \ref{proposition-elliptic-curves}]
We distinguish two cases, depending on whether $\R X$ is non-empty or not. Recall that we are assuming that $e \coloneqq \gcd(r,d) = 2$.


\textbf{Case 1.} Assume that $\R X \neq \emptyset$ (i.e. $n>0$). We assume in this case that $x_0 \in \R X$. Since $x_0 \in \R X$, from formula (\ref{sum-elliptic}), we see that $+_{x_0}$ commutes with $\sigma$. For all $d$, we have an isomorphism of real algebraic varieties 
\begin{equation}
\label{isom-ellip}
\begin{array}{rcl}
    (X,\sigma) & \lra & \big(\Pic_{d}(X),\sigma\big) \\
    p & \lmt & \mathcal{O}_X(p-x_0) \otimes \mathcal{O}_X(d x_0) 
\end{array}
\end{equation} 
Through this identification, the map \cref{map-elliptic-real} corresponds to the morphism of real algebraic varieties
$$
\begin{array}{rcl}
  \theta:\Sym^2(X)   &  \lra  & X \\
  \left[(x,y)\right] &  \lmt  & x+_{x_0} y.
\end{array}
$$ 
Recall that we denote by $n>0$ the number of connected components of $\R X$.

\smallskip

\textbf{Sub-case 1.} If $n=1$, then $\R \Pic_{d}(X) = \R X$ is connected. We then need to show that $\R \Sym^2(X)$ is connected too. Notice that, for a point $[(x,y)] \in \R \Sym^2(X)$, we have either $y=\sigma(x)$ or $x,y \in \R X$. Let $\psi_1:X \to \Sym^2(X)$ be $\psi_1(x)=[(x,\sigma(x))]$ and $\psi_2:\Sym^2(\R X) \to \Sym^2(X)$ the canonical inclusion. We deduce that 
$
\R \Sym^2(X)=\Imm(\psi_1) \cup \Imm(\psi_2).
$
Notice now that $\Imm(\psi_1)$ and $\Imm(\psi_2)$ are both connected being the continuous image of a connected space and have non-empty intersection, given by $[(x,x)] \in \Sym^2(X)$ with $x \in \R X$. We deduce that $\R \Sym^2(X)$ is connected.

\smallskip
   
\textbf{Sub-case 2.} If $n = 2$, then $\R \Pic_{d}(X) = \R X$ has two connected components. Put $\R X=X_0 \bigsqcup X_1$, where $X_0$ is the connected component of $x_0$ and $X_1$ is the other connected component. We need to show that $c^{-1}(X_i) \cap \R \Sym^2(X)$ is connected for $i=0,1$. Notice that for $x \in X_1$ and $y \in X_0$, we have $x+_{x_0}y \in X_1$, while for $x,x'\in X_1$ or $y,y'\in X_0$, we have $x+_{x_0}x', y+_{x_0}y'\in X_0$. It is also possible to check that, if $x \in X \setminus \R X$, we have $x+_{x_0}\sigma(x) \in X_0.$ We deduce from it that $c^{-1}(X_1) \cap \R \Sym^2(X)=\Imm(\psi)$, where $\psi:X_0 \times X_1 \to \R \Sym^2(X) $ is given by $\psi((x,y))=[(x,y)]$ and it is therefore connected since it is the continuous image of a connected space. To prove that  $c^{-1}(X_0) \cap \R \Sym^2(X)$ is also connected, recall first that $X \setminus \R X$ has two connected components when $n=2$ (see for instance \cite{GH}), which we denote by $X' \bigsqcup X''=X \setminus \R X$. Moreover, $Y\coloneqq X' \cup X_0 \cup X_1$ is a connected orientable surface with boundary $\partial X = X_0 \sqcup X_1$. Let $\psi_0:\Sym^2(X_0) \to \R \Sym^2(X)$ and $\psi_1:\Sym^2(X_1) \to \R \Sym^2(X)$ be the natural inclusions and let $\psi_2:Y \to \R \Sym^2(X)$ be the map given by $\psi(y)=[(y,\sigma(y))]$. In a similar way to what was previously remarked, we see that $\Imm(\psi_0) \cap \Imm(\psi_2) \neq \emptyset$ and $\Imm(\psi_1) \cap \Imm(\psi_2) \neq \emptyset$. As the subspaces $\Imm(\psi_0)$, $\Imm(\psi_1)$ and $\Imm(\psi_2)$ are all connected and $c^{-1}(X_0) \cap \R \Sym^2(X)=\Imm(\psi_0) \cup \Imm(\psi_1) \cup \Imm(\psi_2)$, we have proven that $c^{-1}(X_0) \cap \R \Sym^2(X)$ is connected.

\smallskip

\textbf{Case 2.} We now treat the case when $\R X=\emptyset$. Note that, in this case, the isomorphism (\ref{isom-ellip}) is not defined over the reals, since we cannot pick $x_0 \in \R X$.  We first distinguish two cases to describe $\R\Pic_m(X)$: $m=2m'+1$ and $m=2m'$.

\smallskip

\textbf{Sub-case 1.} If $m = 2m'+1$, we can still find an isomorphism of real varieties $(X,\sigma) \simeq (\Pic_{2m'+1}(X),\sigma)$. Namely, we can set
    $$
    \begin{array}{rcl}
      \psi:(X,\sigma) & \lra & \big(\Pic_{2m'+1}(X),\sigma\big) \\
        p & \lmt & \mathcal{O}_X(p) \otimes \mathcal{O}_X(m'(x_0+ \sigma(x_0)\big)).
    \end{array}
    $$
    In particular, $\R\Pic_{2m'+1}(X) = \emptyset$ in this case.
\smallskip

\textbf{Sub-case 2.} If $m = 2m'$, let us consider the map 
$$
\begin{array}{rcl}
\psi':X & \lra & \Pic_{2m'}(X)\\
 p &\lmt & \mathcal{O}_X\big(p-x_0\big) \otimes \mathcal{O}_X\Big(m'\big(x_0+\sigma(x_0)\big)\Big).
\end{array}
$$
Then $\psi'$ is an isomorphism of complex algebraic varieties that does not commute with the real structures of $X$ and $\Pic_{2m'}(X)$. However, we can use $\psi$ to transport the real structure $\si : \Pic_{2m'}(X) \to \Pic_{2m'}(X)$ back to $X$ by setting $\si' \coloneqq (\psi')^{-1} \circ \si \circ \psi'$. Then $\psi' : (X,\si') \to (\Pic_{2m'}(X),\si)$ is an isomorphism of real algebraic varieties. Note that $X^{\si'} \neq \emptyset$, since by \cref{modular_interp_real_locus_no_real_pts_on_curve}, $\R\Pic_{2m'}(X)$ is non-empty. As a matter of fact, the explicit expression for $\si'$ is 
\begin{equation}
    \label{sigma'_direct_def}
    \si'(x) \coloneqq \si(x) -_{x_0} \si(x_0).
\end{equation}
If one wants to check directly that $\si'$ defined as in \cref{sigma'_direct_def} satisfies $(\si')^2 = \mathrm{Id}_X$, $\si'(x_0) = x_0$, and $\psi'\circ\si' = \si \circ \psi$, it suffices to observe the following two facts:
\begin{enumerate}
    \item Using \cref{sum-elliptic}, we have, for all $x,y \in X$,
\begin{equation}
\label{equation-sigma}
    \sigma(x)+_{x_0} \sigma(y) = \sigma(x+_{x_0}y)+_{x_0}\sigma(x_0).
\end{equation}
    \item Using \cref{equation-sigma}, we have, for all $x\in X$, 
    $
    \sigma(x) +_{x_0} \sigma(-_{x_0}x) -_{x_0}\sigma(x_0) = \si(x_0).
    $
\end{enumerate}

\smallskip

\noindent We then have two cases to study for the map $\theta: \Sym^2(\Pic_{d'}(X)) \to \Pic_{2d'}(X)$ in \cref{map-elliptic-real}. Recall from \cref{modular_interp_real_locus_no_real_pts_on_curve} that, since $X$ has genus $1$, $\pi_0(\R\Pic_{2m'}(X)) = \{\R,\H\}$ while $\R\Pic_{2m'+1}(X) = \emptyset$.

\begin{itemize}
    \item $d' = 2m'+1$. In this case, $\R\Pic_{d'}(X) = \R X = X^\si =\emptyset$ and we have a surjective morphism $\psi:\R\Pic_{d'}(X) \to \R \Sym^2(\Pic_{d'}(X))$ given by $\psi(x)=[(x,\sigma(x))]$. In particular, $\R \Sym^2(\Pic_{d'})(X)$ is connected, being the image of a connected space by a continuous function. We deduce from it that $\theta^{-1}(\R) \cap \R \Sym^2(\Pic_{d'}(X))=\R\Sym^2(\Pic_{d'}(X))$ is connected and that $\theta^{-1}(\H) \cap \R \Sym^2(\Pic_{d'}(X))$ is empty.
    \item $d'=2m'$. In this case, we have seen above that both $(\Pic_{d'}(X),\sigma)$ and $\Pic_{2d'}(X)$ can be identified with $(X,\sigma')$, where $\sigma'$ is the map defined in \cref{sigma'_direct_def}, and the map $\theta$ is thus identified, as a real map, with 
    $$
    \begin{array}{rcl}
    \theta:\big(\Sym^2(X),\sigma'\big) & \lra & (X,\sigma')\\
    (x,y) & \lmt & x+_{x_0}y .
    \end{array}
    $$ Since $X^{\sigma'}\neq \emptyset$, we have already shown in Case 1 that, for each connected component $X_i$ of $X^{\sigma'}$, the subspace $\theta^{-1}(X_i) \cap \Sym^2(X)^{\sigma'}$ is connected.
\end{itemize}
\end{proof}

We summarize our results on Higgs bundles over real elliptic curves as follows.

\begin{theorem}
    Let $(X,\sigma)$ be a real curve of genus $g=1$. If $\gcd(r,d) = 1$ or $2$, then the open-closed decomposition in \cref{fibers-topologicaltype} is the decomposition of $\R\MD(r,d)$ into connected components. In particular, \cref{Dol_CC_when_RX_ne,Dol_CC_when_RX_empty} still hold in this case.
\end{theorem}

\section{Character varieties for Klein surfaces}
\label{CharVarSection}

\subsection{Galois action on Betti moduli spaces}

The nonabelian Hodge correspondence establishes a homeomorphism between the moduli space $\MD(r,0)$ of semi\-stable Higgs bundles of rank $r$ and degree $0$ on $X$ and the moduli space $\MB(r,0) \coloneqq \Hom\big(\pi_1 X, \GL_r(\C)\big) /\negmedspace/ \GL_r(\C)$ of $r$-dimensional linear representations of the fundamental group of $X$ (\cite{Simpson_LocalSystems}). When $d \neq 0$ (\cite{NS,AB,FS}), the discrete group $\pi_1 X$ can be replaced by a central extension of $\pi_1 X$ by $\Z$, denoted by $\pi_d$ in what follows, and we still get a homeomorphism $\MD(r,d) \simeq \MB(r,d)$. More precisely, for all $d \in \Z$, let $L_d$ be a $\mathscr{C}^\infty$ line bundle of degree $d$ on $X$. Such a line bundle is unique up to isomorphism of smooth complex vector bundles and it determines a fundamental group 
$\pi_d \coloneqq \pi_1( L_d \setminus \{0\})$
where $\{0\}$ denotes the image of the zero section of $L_d$. The space $L_d \setminus \{0\}$ has the homotopy type of a circle bundle over $X$, making $\pi_d$ a central extension
$0 \longrightarrow \Z \longrightarrow \pi_d \longrightarrow \pi_1 X \longrightarrow 1$
whose isomorphism class is entirely determined by $d \in H^2(X,\Z) \simeq H^2(\pi_1 X, \Z)$. A presentation of $\pi_d$ by generators and relations has been found by Seifert (see for instance \cite{FS}):
$$\pi_d = \left< a_1, b_1, ..., a_g, b_g, c \ \Big|\ \left(\prod_{i=1}^g [a_i, b_i]\right) = c^d, \forall i, [a_i,c] = [b_i, c] = 1 \right>.$$
In particular, $\pi_0 \simeq \Z \times \pi_1 X$ is the trivial central extension of $\pi_1 X$ by $\Z$. To extend the nonabelian Hodge correspondence to the case when $d \neq 0$, the relevant character variety is
$
\MB(r,d) \coloneqq \Hom^\Z \big(\pi_d, \GL_r(\C)\big) /\negmedspace/ \GL_r(\C)
$
where $\Hom^\Z (\pi_d, \GL_r(\C))$ is the space of group morphisms $\rho : \pi_d \to \GL_r(\C)$ satisfying $\rho(c) = e^{i \frac{2\pi}{r}} I_r$, meaning that the following diagram commutes.
\begin{equation}\label{central_ext_rep}
    \begin{tikzcd}
    0 \ar[r] & \Z \ar[r] \ar[d, "n \mapsto e^{i\frac{n2\pi}{r}} I_r"] & \pi_d \ar[r] \ar[d, "\rho"] & \pi_1 X \ar[r] \ar[d] & 1 \\
    0 \ar[r] & \C^* \ar[r] & \GL_r(\C) \ar[r] & \mathbf{PGL}_r(\C) \ar[r] & 1
    \end{tikzcd}
\end{equation}
When $d= 0$, we have $\Hom^\Z(\Z \times \pi_1 X, \GL_r(\C)) \simeq \Hom(\pi_1 X, \GL_r(\C))$, so $\MB(r,0)$ is indeed the usual Betti moduli space in this case. Recall that the orbifold fundamental group of $X/\Gamma$ fits into the short exact sequence
$1 \longrightarrow \pi_1 X \longrightarrow \pi_1 X/\Gamma \longrightarrow \Gamma \longrightarrow 1$
where $\Gamma \coloneqq \Gal(\C/\R) = \{1,\sigma\}$. In particular, there is a well-defined group morphism $\Gamma \to \Out(\pi_1 X)$. If we let $\Gamma$ act on $\Z = Z(\pi_d)$ via $\sigma(n) = -n$, then, since $\pi_d$ is a central extension of $\pi_1 X$ by $\Z$, we can construct, for each $d \in \Z = H^2(\pi_1 X, Z(\pi_d))$, a group morphism $\Gamma \to \Out(\pi_d)$. Combining this group morphism $\Gamma \to \Out(\pi_d)$ with the Galois action $\Gamma \to \Aut(\GL_r(\C))$ given by $\sigma(g) = \overline{g}$, we get a group morphism
\begin{equation}\label{outer_Galois}
\Gamma \longrightarrow \Out\big(\pi_d\big) \times \Out\big(\GL_r(\C)\big)
\end{equation}
which we can use to induce an action of $\Gamma$ on $\MB(r,d)$.

\begin{prop}\label{Galois_action_Betti}
    The group morphism \eqref{outer_Galois} induces an action of $\Gamma$ on the Betti moduli space
    $\MD(r,d) \coloneqq \Hom^\Z(\pi_d, \GL_r(\C)) /\negmedspace/ \GL_r(\C)$.
\end{prop}

\begin{proof}
The group morphism \eqref{outer_Galois} induces an action of $\Gamma \coloneqq \Gal(\C/\R)$ on the orbit space $\Hom(\pi_d,\GL_r(\C)\big) / \GL_r(\C)$, taking closed orbits to closed orbits. Moreover, the condition $\rho(c) = e^{i\frac{2\pi}{r}} I_r$ is preserved by the Galois action because $\sigma\in\Gamma$ acts on $c\in Z(\pi_d)$ via $\sigma(c) = c^{-1}$ and on $e^{i\frac{2\pi}{r}}I_r\in\GL_r(\C)$ via $\sigma(e^{i\frac{2\pi}{r}}I_r) = \overline{e^{i\frac{2\pi}{r}}}I_r = e^{-i\frac{2\pi}{r}}I_r$.
\end{proof}

\subsection{Real locus}

When $\gcd(r,d) = 1$, the Betti moduli space $\MB(r,d)$ is in fact the orbit space 
$\MB(r,d) = \Hom^\Z\big( \pi_d, \GL_r(\C)\big) / \GL_r(\C)$
and in this case we can count the connected components of its real locus $\R\MB(r,d)$. Indeed, for all $r$ and $d$, the nonabelian Hodge correspondence
$\MD(r,d) \simeq \MB(r,d)$
is $\Gamma$-equivariant with respect to the Galois actions on each side, defined respectively in Equation \eqref{Galois_action_Dolbeault} and Proposition \ref{Galois_action_Betti} (see for instance \cite[\S 3]{ALS}). So, when $\gcd(r,d) = 1$, we can use the results of Section \ref{description-fibers} to count the connected components of $\R\MB(r,d)$, which is the content of Theorem \ref{CC_MBrd}.

\begin{theorem}\label{CC_MBrd}
Let $\Gamma \coloneqq \Gal(\C/\R)$ act on the Betti moduli space $\MB(r,d)$ via the action induced by the group morphism \eqref{outer_Galois}.
\begin{enumerate}
    \item If $\R X \neq \emptyset$ and $\gcd(r,d) = 1$, then $\R\MB(r,d)$ has $2^{n-1}$ connected components, where $n>0$ is the number of connected components of $\R X$.
    \item If $\R X = \emptyset$ and $\gcd(r,d) = 1$, then:
    \begin{itemize}
        \item If $d = 2 d'$ and $r(g-1) \equiv 1\ (\mathrm{mod}\ 2)$, then $\R \MB(r,d)$ has one connected component.
        \item If $d = 2 d'$ and $r(g-1) \equiv 0\ (\mathrm{mod}\ 2)$, then $\R \MB(r,d)$ has two connected components.
        \item If $d = 2 d'+1$ and $r(g-1) \equiv 1\ (\mathrm{mod}\ 2)$, then $\R \MB(r,d)$ has one connected component.
        \item If $d = 2 d'+1$ and $r(g-1) \equiv 0\ (\mathrm{mod}\ 2)$, then $\R \MB(r,d)$ is empty.
    \end{itemize}
\end{enumerate}
\end{theorem}

\begin{proof}
By $\Ga$-equivariance of the nonabelian Hodge correspondence $\MD(r,d) \simeq \MB(r,d)$, we can view the map $c : \R\MD(r,d) \to \pi_0(\R\Pic_d(X))$ from \eqref{obstruction_map} as a map $c : \R\MB(r,d) \to \pi_0(\R\Pic_d(X))$.
We then have an open-closed decomposition 
$$\R\MB(r,d) = \bigsqcup_{s\in \pi_0(\R\Pic_d(X))} \R\MB(r,d,s),$$
analogous to \eqref{fibers-topologicaltype}, with $\R\MB(r,d,s) \coloneqq c^{-1}(s)$. Since $\gcd(r,d) =1$, Theorem \ref{CC_real_locus} shows that the subspaces $\R\MB(r,d,s) \simeq \R\MD(r,d,s)$ are all connected. The count of non-empty such subspaces then follows immediately from the case analysis conducted in Section \eqref{description-fibers}. Note that the case when both $r = 2r'$ and $d=2d'$ also occurs in Sections \ref{modular_interp_real_locus_with_real_pts_on_curve} and \ref{modular_interp_real_locus_no_real_pts_on_curve}, but does not occur here because we are assuming that $\gcd(r,d) =1$.
\end{proof}

We can now give a representation-theoretic modular interpretation of the connected components $\R\MB(r,d,s)$. Recall first that, given a group morphism $\phi : \Gamma \to \Out(\pi_d)$, the extensions of $\Gamma$ by $\pi_d$ inducing this morphism are parameterized by $H^2(\Gamma, Z(\pi_d))$, where the action of $\Gamma$ on $Z(\pi_d)$ is the one induced by $\phi$. Combined with the next lemma, this shows that the integer $d\in\Z$ uniquely determines such an extension.

\begin{lem}\label{real_Seifert_group}
    Let $\Gamma\coloneqq \Gal(\C/\R)=\{1,\sigma\}$ act on the centre $Z(\pi_d) \simeq \Z$ of $\pi_d$ via the involution $\sigma(n) = -n$. Then $H^2(\Gamma; Z(\pi_d)) = \{0\}$.
\end{lem}

\begin{proof}
   By definition of $H^2$ in group cohomology, for $\Gamma = \{1, \sigma\}$ acting on $\Z$ via $\sigma(n) = -n$, we have $H^2(\Gamma,\Z) = \Z^\Gamma / \{n + \sigma(n)\ |\ n \in \Z\} = \{0\}.$
\end{proof}

So let us call $\pi_d^\Ga$ the unique extension 
$1 \lra \pi_d \lra \pi_d^\Ga \lra \Ga \lra 1$
determined by Lemma \ref{real_Seifert_group}. The abstract group $\pi_d^\Ga$ thus defined also fits in the following short exact sequence
\begin{equation}\label{realSeifertGroup_as_non_central_ext}
    0 \lra \Z \lra \pi_d^\Ga \lra \pi_1(X/\Ga) \to 1,
\end{equation}
which is a non-central extension of $\pi_1(X/\Ga)$ by $\Z$ (since $\pi(X/\Ga)$ acts non-trivially on $\Z$). Geometrically, extensions of the form \eqref{realSeifertGroup_as_non_central_ext} arise when fixing a real structure on the smooth complex line bundle $L_d$. For fixed $d$, isomorphism classes of such extensions are classified by the real invariants of $L_d$ but, as an abstract group, $\pi_d^\Ga$ is entirely determined by $d$ (see \cite[Remark 2.8]{Sch_JDG}). Since $\gcd(r,d) =1$, we can decompose $\R\MB(r,d)$ as in \eqref{decompositionintro1} and write
$\R\MB(r,d) = \MB^\R(r,d) \sqcup \MB^\H(r,d)$.
By \cite[Theorem 4.5]{Sch_JDG}, the points in $\MB^\R(r,d)$ correspond to so-called real representations of $\pi_d$, meaning group morphisms $\rho : \pi_d^\Ga \to \GL_r(\C) \rtimes_{\si} \Ga$ that make the diagram \ref{central_ext_rep} as well as the following commute.
\begin{equation}\label{real_rep}
    \begin{tikzcd}
    1 \ar[r] & \pi_d \ar[r] \ar[d] & \pi_d^\Ga \ar[r] \ar[d, "\rho"] & \Ga \ar[r] \ar[-,double line with arrow={-,-}]{d} & 1 \\
    1 \ar[r] & \GL_r(\C) \ar[r] & \GL_r(\C) \rtimes_{\si} \Ga \ar[r] & \Ga \ar[r] & 1
    \end{tikzcd}
\end{equation}
Here, in the semi-direct product $\GL_r(\C) \rtimes_{\si} \Ga$, the Galois group $\Ga$ acts on $\GL_r(\C)$ via $\si(g) = \ov{g}$. The quaternionic case $\MB^\H(r,d)$ is similar except that, this time, one represents $\pi_d$ in the (unique) non-trivial extension $\GL_r(\C) \times_\H \Ga$ determined by the element $\H$ in the Brauer group $\mathrm{Br}(\R) = H^2(\Ga,Z(\GL_r(\C)) \simeq \Z/2\Z$. Taking into account the equivalence relation between such representations, we obtain the following result.

\begin{theorem}\label{orbifold_Betti}
    Let $\Gamma \coloneqq \Gal(\C/\R)$ act on the Betti moduli space $\MB(r,d)$ via the action induced by the group morphism \eqref{outer_Galois} and assume that $\gcd(r,d) = 1$. Then the real locus of the Betti moduli space admits the following open-closed decomposition.
    $$
    \R\MB(r,d) = \Hom^\Z_\Ga\big(\pi_d^\Ga, \GL_r(\C) \rtimes_{\si} \Gamma\big) / \GL_r(\C) \sqcup \Hom^\Z_\Ga\big(\pi_d^\Ga, \GL_r(\C) \times_\H \Gamma\big) / \GL_r(\C)
    $$
    where $\Hom^\Z_\Ga(\pi_d^\Ga, \GL_r(\C) \rtimes_{\si} \Gamma)$ means the set of group morphisms $\rho : \pi_d^\Ga \to \GL_r(\C) \rtimes_{\si} \Ga$ that make diagram \eqref{real_rep} commute and such that $\rho|_{\pi_d}$ makes diagram \eqref{central_ext_rep} commute, and similarly for $\Hom^\Z_\Ga(\pi_d^\Ga, \GL_r(\C) \times_\H \Gamma)$.
\end{theorem}

Note that the coprimality assumption $\gcd(r,d) = 1$ implies that $\MB^{\H}(r,d) = \emptyset$ when $\R X \neq \emptyset$. So by Theorem \ref{CC_MBrd}, the space of real representations $\MB^\R(r,d)$ has $2^{n-1}$ connected components in this case. By  \cite[Proposition 2.9]{Sch_JDG}, we can then view the real Betti moduli space $\MB^\R(r,d)$ as the moduli space of $\Ga$-equivariant representations $\rho : \pi_d \to GL_r(\C)$. More precisely, by choosing a real point of $L \setminus \{0\}$ as a base point for $\pi_d \coloneqq \pi_1(L \setminus \{0\})$, we can set up a homeomorphism
$$
 \Hom^\Z_\Ga\big(\pi_d^\Ga, \GL_r(\C) \rtimes_{\si} \Gamma\big) / \GL_r(\C) \simeq  \Hom^\Z\big(\pi_d, \GL_r(\C)\big)^\Ga / \GL_r(\R)
$$
because the choice of such a base point determines an isomorphism of extensions $\pi_d^\Ga \simeq \pi_d \rtimes_{\si} \Ga$, allowing us to replace representations of the orbifold fundamental group $\pi_d^\Ga$ by $\Ga$-equivariant representations of $\pi_d$ (see \cite[Proposition 2.9]{Sch_JDG}). The point is that this gives a more direct description of the fibres of the map $c : \R\MB(r,d) \to \pi_0(\R\Pic_d(X))$, by looking at the sign of $\det(\rho(\tilde{c}))$ for $c\in \pi_1 X$ a $\sigma$-fixed loop in $X$ and $\tilde{c}\in\pi_d$ an arbitrary lift of $c$ to $\pi_d$ (see \cite[Theorem 4.8]{Sch_JDG}).

\subsection{An application to twisted complex character varieties}
\label{twistedCharVarSection}
We note that the real Betti moduli space
$
\Hom^\Z_\Ga\big(\pi_d^\Ga, \GL_r(\C) \rtimes_{\si} \Gamma\big) / \GL_r(\C)
$
does not admit a natural structure of algebraic variety, because $\Ga$ acts on $\GL_r(\C)$ via $\si(g) = \ov{g}$. However, when $\gcd(r,d) = 1$ and $\R X \neq \emptyset$, it is homeomorphic to the affine complex algebraic variety 
$
\Hom^\Z_\Ga\big(\pi_d^\Ga, \GL_r(\C) \rtimes_{\widetilde{\si}} \Gamma\big) / \GL_r(\C)  
$
where $\Ga$ acts on $\GL_r(\C)$ via the Cartan involution $\widetilde{\si}(g) = \,^t{g}^{-1}$. To see it, it suffices to notice that the real structures
$\sigma (E,\phi) = (\sigma^*(\overline{E}),\sigma^*(\overline{\phi}))$ and $\widetilde{\si} (E,\phi) = (\sigma^*(\overline{E}),-\sigma^*(\overline{\phi}))$ of $\MD(r,d)$ are conjugate via the order 4 automorphism
\begin{equation}\label{conjugating_the_Galois_action}
\gamma: 
\begin{array}{rcl}
\MD(r,d) & \longrightarrow & \MD(r,d) \\
(E,\phi) & \longmapsto & (E,i\phi) .
\end{array}
\end{equation}
of the Dolbeault moduli space. Indeed, 
\begin{equation}\label{conjugating_the_Galois_action_pf}
\big( \widetilde{\si} \circ \gamma \big) \big( E,\phi \big) = \big( \si^*(\ov{E}), -\si^*(\ov{i\phi}) \big) =  \big( \si^*(\ov{E}), i \si^*(\ov{\phi}) \big) = \big( \gamma \circ \si \big) \big( E,\phi \big).
\end{equation}
Since $\widetilde{\si}\circ\gamma = \gamma \circ \si$, the map $\gamma$ induces an isomorphism of real algebraic varieties
$\MD(r,d)^\si \simeq \MD(r,d)^{\widetilde{\si}}$.
Note now that, on the Betti side of the nonabelian Hodge correspondence, the real structure $\si(E,\phi) = (\si^*(\ov{E}), \si^*(\ov{\phi}))$ corresponds to the anti-holomorphic involution $\beta_{\si}([\rho]) = [\si\circ\rho\circ\si_*]$
where $\si_* : \pi_d \to \pi_d$ is defined up to inner automorphism and $\sigma(g) = \ov{g}$ on $\GL_r(\C)$, while the real structure $\widetilde{\si}(E,\phi) = (\si^*(\ov{E}), -\si^*(\ov{\phi}))$ corresponds to the \textit{holomorphic} involution $\beta_{\widetilde{\si}}([\rho]) = [\widetilde{\si}\circ\rho\circ\si_*]$
where $\widetilde{\si}(g) = \,^t g^{-1}$ on $\GL_r(\C)$, as studied for instance in \cite{Sco}. Note that the involution $\beta_{\widetilde{\si}}$ is well-defined for the same reasons that $\beta_{\si}$ is well-defined. Indeed, like the involution $\si(g) = \ov{g}$, the involution $\widetilde{\si}(g)=\,^tg^{-1}$ also sends the element $e^{i\frac{2\pi}{r}} I_r$ to its inverse, so we can reason exactly as in \cref{Galois_action_Betti} to show that $\beta_{\widetilde{\si}}$ is well-defined. Therefore, there is an induced homeomorphism
$
\Fix(\beta_\si) \simeq \Fix(\beta_{\widetilde{\si}})
$
which is best understood in the context of branes coming from real structures on the moduli space of Higgs bundles. Namely, with respect to the hyperkähler structure of the Hitchin moduli space of harmonic bundles, $\Fix(\beta_\si)$ is an $(A,A,B)$-brane, while $\Fix(\beta_{\widetilde{\si}})$ is an $(A,B,A)$-brane. We refer to  \cite{BaSc1, BaSc2, BiGP} for further details on this. The point now is that, when $\gcd(r,d) = 1$ and $\R X \neq \emptyset$, the real locus $\MD(r,d)^\si$ consists, as we have seen, only of real Higgs bundles. So, on the Betti side, we have an identification
$
\Fix(\beta_\si) =
\Hom^\Z_\Ga\big(\pi_d^\Ga, \GL_r(\C) \rtimes_{\si} \Gamma\big) / \GL_r(\C)\ .
$
As a consequence, we also have
$
\Fix(\beta_{\widetilde{\si}}) =
\Hom^\Z_\Ga\big(\pi_d^\Ga, \GL_r(\C) \rtimes_{\widetilde{\si}} \Gamma\big) / \GL_r(\C)\ .
$
Combining this with Theorem \ref{CC_MBrd}, we obtain the following result on the topology of the affine complex algebraic variety $\Hom^\Z_\Ga(\pi_d^\Ga, \GL_r(\C) \rtimes_{\widetilde{\si}} \Gamma) / \GL_r(\C)$. 

\begin{theorem}\label{cxOrbifoldBetti}
    Let $\Gamma \coloneqq \Z/2\Z$ act algebraically on the Betti moduli space $\MB(r,d)$ via the group morphism
    $\Ga \lra \Out(\pi_d) \times \Out(\GL_r(\C))$
    induced by the real structure of $X$ and the Cartan involution $\widetilde{\si}(g) = \,^t g^{-1}$ of $\GL_r(\C)$. If $\gcd(r,d) =1$ and $\R X \neq \emptyset$, then the fixed-point of this $\Ga$-action is isomorphic to the twisted character variety
    \begin{equation}\label{cx_twisted_char_var}
    \Hom^\Z_\Ga\big(\pi_d^\Ga, \GL_r(\C) \rtimes_{\widetilde{\si}} \Gamma\big) / \GL_r(\C)
    \end{equation}
    and has $2^{n-1}$ connected components.
\end{theorem}

Note that it is not obvious, at the outset, that the complex character variety \eqref{cx_twisted_char_var}, coming from an \textit{orientation-reversing} involution $\sigma$ on the underlying topological surface of the Riemann surface $X$ and a \textit{holomorphic} involution $\widetilde{\si}$ on the complex Lie group $\GL_r(\C)$, can be studied by means of real algebraic geometry, changing the Cartan involution $\widetilde{\si}(g) = \,^t g^{-1}$ to the split real form $\sigma(g) = \ov{g}$ of $\GL_r(\C)$. But the point is that the self-homeomorphism of $\MB(r,d)$ conjugating the involutions $\beta_\si$ and $\beta_{\widetilde{\si}}$ is defined via the nonabelian Hodge correspondence, and that on the Dolbeault side it relies on the $\C^*$-action on $\MD(r,d)$ to provide the automorphism $\gamma$ in \eqref{conjugating_the_Galois_action}. And once this is done, Theorem \ref{cxOrbifoldBetti} becomes a simple reformulation of Theorem \ref{CC_MBrd}, which itself follows from our main result (\cref{CC_real_locus}).


\end{document}